\documentclass[12pt]{amsart}
\usepackage[a4paper, hmargin={2cm,2cm},vmargin={3cm,2.5cm}]{geometry}
\usepackage{amsmath}
\usepackage{amsfonts}
\usepackage{amssymb,color}
\usepackage[hidelinks]{hyperref}
\usepackage{amsthm,doi}
\usepackage{thmtools}
\usepackage{tikz}
\usepackage{enumitem}
\usepackage{pgfplots}
\pgfplotsset{compat=1.18} 
\usepackage[showonlyrefs]{mathtools}
\mathtoolsset{showonlyrefs=true}

\DeclareMathOperator{\var}{Var}
\DeclareMathOperator{\cov}{Cov}

\numberwithin{equation}{section}

\newtheorem{lemma}{Lemma}[section]
\newtheorem{theorem}[lemma]{Theorem}
\newtheorem{rem}[lemma]{Remark}
\newtheorem{coro}[lemma]{Corollary}
\newtheorem{prop}[lemma]{Proposition}

\newcommand{\Nz}{\mathcal{N}^{\mathrm{z}}_\rho}
\newcommand{\Nc}{\mathcal{N}^{\mathrm{c}}_\rho}
\newcommand{\Ncp}{\mathcal{N}^{\mathrm{c},+}_\rho}
\newcommand{\Ncm}{\mathcal{N}^{\mathrm{c},-}_\rho}

\DeclareMathOperator{\jac}{Jac}

\newcommand{\Ordo}{\mathrm{O}}
\newcommand{\OrdoP}{\mathrm{O}_{\mathbb{P}}}
\newcommand{\E}{\mathbb{E}}

\title[Local repulsion between zeros and critical points of the GEF]{Local repulsion between zeros and critical points of the Gaussian Entire Function}



\author[A. Haimi]{Antti Haimi}
\address[A. Haimi]{Faculty of Science and Engineering, Åbo Akademi University,
Tuomiokirkontori 3
20500 Turku, Finland}
\email{antti.haimi@abo.fi}

\author[L. Odelius]{Lukas Odelius}
\address[L. Odelius]{Faculty of Mathematics, University of Vienna, Oskar-Morgenstern-Platz 1, A-1090 Vienna, Austria}
\email{lukas.odelius@univie.ac.at}

\author[J. L. Romero]{Jos\'{e} Luis Romero}
\address[J. L. Romero]{Faculty of Mathematics, University of Vienna, Oskar-Morgenstern-Platz 1, A-1090 Vienna, Austria, and Acoustics Research Institute, Austrian Academy of Sciences, Dr. Ignaz Seipel-Platz 2,	AT-1010 Vienna, Austria}
\email{jose.luis.romero@univie.ac.at}

\thanks{This research was funded in whole or in part by the Austrian Science Fund (FWF): 10.55776/Y1199. For open access purposes, the authors have applied a CC BY public copyright license to any author-accepted manuscript version arising from this submission.}

\begin{document}
\begin{abstract}
We study the zeros and critical points of different indices of the standard Gaussian entire function on the complex plane (whose zero set is stationary). We provide asymptotics for the second order correlations of all the corresponding number statistics on small observation disks, showing various rates of local repulsion. The results have consequences for signal processing, as they show extremely strong repulsion between the local maxima and zeros of spectrograms of noise computed with respect to Gaussian windows.
\end{abstract}
\maketitle
\section{Introduction and results}
\subsection{Results}
We study the zeros of the random \emph{Gaussian entire function} (GEF)
\begin{align}\label{eq_gef}
G(z) = \sum_{n=0}^\infty \frac{\xi_n}{\sqrt{n!}} z^n, \qquad z \in \mathbb{C},
\end{align}
where $\xi_n$ are independent standard complex random variables \cite{MR1783614, gafbook, NSwhat}, and its \emph{covariant derivative}
\begin{align}\label{eq_F}
F(z) = \bar{\partial}^* G(z) = \bar{z} G(z) - \partial G(z).
\end{align}
The zeros of $F(z)$ are called \emph{critical points} of $G$,
and are instrumental in the analysis of heuristic or approximate models in string theory \cite{MR2104882}. From the point of view of complex geometry, $G$ is a random holomorphic section to a standard line bundle on the plane and the non-analyticity of $F$ is an effect of the Gaussian metric.

While the conformality of $G$ implies that its zeros have non-negative winding numbers, the critical points of $G$ can be further classified according to their \emph{index}:
\begin{align}\label{eq_index}
\mathrm{Index}(F,z) := \mathrm{sgn} \jac F(z),
\end{align}
where $\jac F(z)$ is the Jacobian determinant of $F$ (considered as a function of two real variables).

Though not obvious, the zeros and critical points of $G$ of a certain index define jointly stationary point processes (see Sections \ref{sec_rel} and \ref{sec_sim}). We shall be interested in the following statistics:
\begin{align}\label{eq_stats}
\begin{aligned}
\Nz &= \#\{z \in \mathbb{C}: G(z)=0, |z|<\rho\},
\\
\Nc &= \#\{z \in \mathbb{C}: F(z)=0, |z|<\rho\},
\\
\Ncp &= \#\{z \in \mathbb{C}: F(z)=0, \,\mathrm{Index}(F,z)=1 ,|z|<\rho\},
\\
\Ncm &= \#\{z \in \mathbb{C}: F(z)=0, \,\mathrm{Index}(F,z)=-1 ,|z|<\rho\}.
\end{aligned}
\end{align}
(Critical points of index zero almost surely do not occur.) The first order moments of \eqref{eq_stats} are
\begin{align}
\mathbb{E} \Nz = \rho^2; \qquad
\mathbb{E} \Nc = \frac{5}{3} \rho^2; \qquad
\mathbb{E} \Ncp = \frac{4}{3} \rho^2; \qquad
\mathbb{E} \Ncm = \frac{1}{3} \rho^2,
\end{align}
see, e. g, \cite{gafbook}, \cite[Corollary 5]{MR2104882} and \cite[Section 6.8]{hkr22}. Our main result concerns local correlations between the quantities \eqref{eq_stats} and reads as follows.
\begin{theorem}\label{th_main}
Let $G$ be the Gaussian entire function \eqref{eq_gef} and $F=\bar{\partial}^*G$ its covariant derivative, cf. \eqref{eq_F}. Then we have the following asymptotics, valid for $0<\rho<1$:
\begin{align}\label{eq_repz}
    \mathbb{E}[\Nz \cdot (\Nz - 1)] \asymp \rho^6,
\end{align}    
\begin{align}\label{eq_repc}
    \mathbb{E}[\Nc \cdot (\Nc -1)] \asymp \rho^4
    \quad\mbox{and}\quad
    \mathbb{E}[\Nc \cdot (\Nc -1)] \sim \frac{6}{25}(\mathbb{E}[\Nc])^2
    \mbox{ as } \rho \to 0^+,
\end{align}
\begin{align}\label{eq_repp}
    \mathbb{E}[\Ncp \cdot (\Ncp - 1)] \asymp \rho^7,
\end{align}
\begin{align} \label{eq_repm}
    \mathbb{E}[\Ncm \cdot (\Ncm - 1)] \asymp \rho^7,
\end{align}
\begin{align} \label{eq_cmcp}
    \mathbb{E}[\Ncm \cdot \Ncp] \asymp \rho^4
    \quad\mbox{and}\quad
    \mathbb{E}[\Ncm \cdot \Ncp] \sim \frac{3}{4} \cdot \mathbb{E}[\Ncm] \cdot \mathbb{E}[\Ncp]
    \mbox{ as } \rho \to 0^+,
\end{align}
\begin{align} \label{eq_zc}
    \mathbb{E}[\Nz \cdot \Nc] \asymp \rho^6,
\end{align}
\begin{align}\label{eq_zcp}
    \mathbb{E}[\Nz \cdot \Ncp]  \asymp \rho^6,
\end{align}
\begin{align} \label{eq_zcm}
    \mathbb{E}[\Nz \cdot \Ncm] \asymp \rho^{20}.
\end{align}
\end{theorem}
(Here and throughout we write $A(\rho) \asymp B(\rho)$ if $A(\rho)/B(\rho)$ is bounded above and below by positive constants, while $A(\rho) \sim B(\rho)$ means that $\lim_{\rho\to0^+} A(\rho)/B(\rho)=1$.) 

The first asymptotic \eqref{eq_repz} is well-known and follows from more precise results \cite{MR1690355, gafbook, MR2885614}. It expresses a strong \emph{local repulsion} between zeros of the GEF $G$. Indeed, it means that the expected number of pairs of distinct zeros
\begin{align}
\Nz \cdot (\Nz - 1) = \# \big\{ (z,w) \in \mathbb{C}^2 \,:\, G(z)=G(w)=0, |z|,|w|<\rho, z\not=w \big\}
\end{align}
to be found in a small disk is asymptotically smaller than the corresponding count for two independent GEF $G^1$, $G^2$:
\begin{align}
\mathbb{E} \# \big\{ (z,w) \in \mathbb{C}^2 \,:\, G^1(z)=G^2(w)=0, |z|,|w|<\rho, z\not=w \big\} = \big(\mathbb{E}\Nz\big)^2 = \rho^4.
\end{align}
To compare, for a standard Poisson process $X$ on the plane, the number statistic $\mathcal{N}_\rho(X)$ satisfies $\mathbb{E} [\mathcal{N}_\rho(X) \cdot (\mathcal{N}_\rho(X)-1)] = \big(\mathbb{E} \mathcal{N}_\rho(X) \big)^2$. In this case, one speaks of a non-repulsive point process.

The asymptotic \eqref{eq_repc} was proved in a slightly different context in \cite{MR2966361} and shows that critical points of GEF exhibit local repulsion albeit in a much more moderate form than zeros of GEF. While the orders of magnitude of $\mathbb{E}[\Nc \cdot (\Nc -1)]$ and $(\mathbb{E} \Nc)^2$ are comparable, as in the Poissonian statistics, the so-called \emph{repulsion factor} is
\begin{align}
\lim_{\rho\to0^+} \frac{\mathbb{E}[\Nc \cdot (\Nc -1)]}{(\mathbb{E}[\Nc])^2}=\frac{6}{25} \in (0,1).
\end{align}
In this case, one often speaks of a \emph{weakly repulsive point process}. 

Our main contribution is \eqref{eq_repp}, $\ldots$, \eqref{eq_zcm}. Though critical points of GEF repel each other weakly, Theorem \ref{th_main} shows that critical points of the same index repel each other very strongly \eqref{eq_repp}, \eqref{eq_repm}, and the overall weak repulsion \eqref{eq_repc} is due to weak interactions among critical points of different indices \eqref{eq_cmcp}.

Second, Theorem \ref{th_main} shows that zeros and critical points of GEF repel each other strongly \eqref{eq_zc}, in fact with the same order as the repulsion among zeros of GEF \eqref{eq_repz}. This provides a quantitative version of one of the results in \cite{MR3937291} (which pertains to the more general context of random holomorphic sections to fiber bundles on complex manifolds). Remarkably, the repulsion order \eqref{eq_zc} is achieved 
by the subclass of critical points of positive index \eqref{eq_zcp}, while \emph{the repulsion between critical points of negative index and zeros of GEF is much more intense} \eqref{eq_zcm}. 

\subsection{Weighted amplitude of GEF}
Theorem \ref{th_main} has a natural reformulation in terms of the \emph{weighted amplitude} of the GEF $G$ \eqref{eq_gef},
\begin{align}\label{eq_S}
S(z) = e^{-\tfrac{1}{2}|z|^2}|G(z)|, \qquad z \in \mathbb{C},
\end{align}
which is a (non-Gaussian) stationary random field on $\mathbb{C}$ (see Section \ref{sec_sim}). By the analyticity of $G$, the zeros of $S$ are exactly its local minima (see, e.g. \cite[Section 8.2.2]{gafbook} or \cite[Lemma 3.1]{efkr24}). Other critical points of $S$ (in the standard Euclidean sense) can be related to $G$ as follows: near a point $z$ where $S$ does not vanish, the gradient of $S$ is related to the covariant derivative of $G$ \eqref{eq_F} by
\begin{align}
|\nabla S(z)| = e^{-\tfrac{1}{2}|z|^2}|F(z)|.
\end{align}
Hence, those critical points of $S$ which are not local minima are exactly the critical points of $G$ (that is, the zeros of $F$), while the zeros of $S$ are of course the zeros of $G$.\footnote{To be precise, the assertion holds after excluding the zero probability event that $G$ and $F$ have common zeros.} In addition, the index of a critical point of $G$ is given by the signature of the Hessian matrix of $S$ and thus classifies the kind of critical point that $S$ has (local maxima or saddle point), see \cite[Section 3.1]{MR2104882}, \cite[Section 9.2]{hkr22}, \cite[Section 9.2]{fhkr24}. In summary, the statistics \eqref{eq_stats} also count the zeros and critical points of $S$ on a small disk of radius $\rho$:
\begin{align}\label{eq_stats_S}
\begin{aligned}
\Nz &= \# \mbox{zeros of S in }B_\rho = \# \mbox{local minima of S in }B_\rho,
\\
\Ncp &= \# \mbox{saddle points of S in }B_\rho,
\\
\Ncm &= \# \mbox{local maxima of S in }B_\rho,
\\
\Nc &= \Ncp + \Ncm.
\end{aligned}
\end{align}
This suggests a partial analogy between our result and the statistics of the critical points of a \emph{stationary Gaussian} function
$f:\mathbb{R}^2 \to \mathbb{R}$, which are the zeros of the \emph{stationary Gaussian} field $\nabla f:\mathbb{R}^2 \to \mathbb{R}^2$. In this context, \cite{ladgham2023local} shows that the correlation of local extrema of $f$ (maxima or minima) on a small disk of radius $\rho$ has order $\rho^7$, that of saddle points has order $\rho^7 \log(1/\rho)$, while extrema and saddle points repel each other weakly, with correlations $\sim r\rho^4$ for some $0<r<1$. 
There is no obvious analogue of \eqref{eq_zcp} and \eqref{eq_zcm} for a general stationary Gaussian field $f$, but a partial analogy is possible with certain isotropic planar Gaussian waves. Indeed, in that context, \cite{MR3947635} shows repulsion of order (at least) $\rho^{12}$ between local maxima and minima and leaves open the problem of determining the exact repulsion order, which in our case is shown to be $\rho^{20}$. Further analogies between the zeros and critical points of $S$ and $f$ seem unclear, as the zeros of $f$ form curves while those of $S$ are discrete, and $S$ is non-negative while $f$ is real-valued.

\subsection{Time-frequency landmarks of noise}
Theorem \ref{th_main} has also applications in the field of \emph{time-frequency analysis}. The \emph{short-time Fourier transform} (with Gaussian window) of a distribution $f \in \mathcal{S}'(\mathbb{R})$ is the two-variable function
\begin{align}
Vf(x, \xi) =\big(\tfrac{2}{\pi}\big)^{\frac14} \int_{\mathbb{R}} f(t) e^{-(t-x)^2} e^{-2i t \xi} \, dt,
\qquad (x,\xi) \in \mathbb{R}^2,
\end{align}
where the integral is interpreted distributionally. The squared magnitude $|Vf(x, \xi)|^2$ is called the \emph{spectrogram} of $f$ and measures the importance of the frequency $\xi$ in the \emph{signal} $f$ at time $t=x$. In statistical signal processing one is often interested in signals contaminated with noise. A powerful recent insight is that the statistics of zeros and critical points of spectrograms help identify time-frequency regions dominated by noise
\cite{flandrin2018explorations, gardner2006sparse, MR1662451, MR4047541, 7180335, 7869100, bh, MIRAMONT2024109250, miramont2024benchmarking, pascal2024point}. Importantly, the spectrogram of the short-time Fourier transform of standard complex (Gaussian) white noise $\mathcal{N}$ can be identified with (the square of) the weighted amplitude \eqref{eq_S} of a Gaussian entire function:
\begin{align}
S^2(z) \stackrel{(d)}{=} |V \mathcal{N} (x-i\xi)|^2, \qquad z=x+i\xi \in \mathbb{C},
\end{align}
see \cite{MR1662451, MR4047541, bh}. As a consequence, the interpretation of the statistics \eqref{eq_stats_S} also holds for $|V \mathcal{N} (x+i\xi)|$, the square root of the spectrogram of complex white noise, in lieu of $S$.

In this light, our results support several heuristics of the signal processing literature \cite[Chapters 10 - 15]{flandrin2018explorations}. For example, spectrogram reassignment is a popular procedure to sharpen spectrograms based on a certain vector field, whose attractors are the spectrogram local maxima and whose repellers are the spectrogram zeros. The \emph{super-repulsion} among these kinds of landmarks expressed by \eqref{eq_zcm} is strongly consistent with the success of spectrogram sharpening \cite[Chapters 12 and 14]{flandrin2018explorations}.

\subsection{Related literature}\label{sec_rel}
There is extensive recent work on the repulsion of critical points of real-valued Gaussian fields. In \cite{MR3947635}, the authors studied the two-point function of the random plane wave, and found that the second factorial moment of the number of critical points in a small disk behaves as the fourth power of the radius. Similar results were obtained in \cite{MR4187723} for general isotropic and stationary Gaussian fields. The technique in these two papers is based on a subtle near-diagonal expansion of certain covariance matrices. In \cite{ladgham2023local}, the authors studied similar questions with a different method based on a refined analysis of random Taylor expansions, and were able to obtain precise constants describing the near diagonal behavior of two-point functions. Yet another approach to second order local statistics of critical points, based on random matrix theory, was used in \cite{MR4426161}. The index \eqref{eq_index} can be interpreted as a winding number, which is an attribute that has also received important attention in the stationary context \cite{buckley2017fluctuations, buckley2018winding}.

The Gaussian entire function \eqref{eq_gef} is not stationary in the usual sense, even after renormalization by variance: $e^{-|z|^2/2} G(z)$. Rather, it posses a special kind of invariance that we call \emph{twisted stationarity} \cite{hkr22} and means that the \emph{Bargmann-Fock shifts}
\begin{align}\label{eq_bfsh}
f(z) \mapsto f_\zeta(z) = e^{-\frac{|\zeta|^2}{2} + z \overline{\zeta}} f(z- \zeta), \qquad \zeta \in \mathbb{C},
\end{align}
preserve the stochastics of $G$.\footnote{This is an instance of what is called \emph{projective invariance} in \cite{NSwhat}.} Importantly, the covariant derivative \eqref{eq_F}, which is not analytic, shares the same stochastic symmetry, a fact that we exploit systematically (see Section \ref{sec_sim} and \cite[Section 2.2]{MR2346279}).

The invariance of $G$ and $F$ under the Bargmann-Fock shifts implies that their zero sets are stationary in the standard sense. Moreover, a more careful argument that takes indices into account shows that the test disk used to compute the statistics $\eqref{eq_stats}$ can be replaced by any other disk of the same radius without affecting the stochastic properties (see Section \ref{sec_sim}). In contrast, the \emph{holomorphic critical points of $G$}, that is, the zeros of the holomorphic derivative $\partial G$ are not a stationary point process, and are indeed quite different from the critical points studied here. Holomorphic critical points arise naturally in the context of random polynomials, and have been studied in \cite{MR3689975} for polynomials with independent zeros and in  \cite{MR3340325} for polynomials with Gaussian coefficients. 

First order statistics of critical points of random GEF have been calculated in \cite{MR2104882} in the more general context of polynomial sections to fiber bundles on complex manifolds. In the same context, second order local correlations among critical points of GEF have been studied in \cite{MR2966361}, while \emph{large scale} second order statistics have been recently derived in \cite{fhkr24}. The correlation between zeros and critical points of GEF has been studied in \cite{MR3937291} also in context of complex manifolds. The asymptotic equivalence \eqref{eq_zc} provides a more quantitative version of one of the results in \cite{MR3937291} (when specialized to the complex plane). 

Our work is greatly inspired by \cite{ladgham2023local}. Though we cannot directly apply their results to the non-stationary functions \eqref{eq_gef}, \eqref{eq_F} or the non-Gaussian function \eqref{eq_S}, we adapt many of the ideas and methods from \cite{ladgham2023local}, and extend them to also treat correlations between zeros and critical points (a setting lacking rotational symmetries, see Section \ref{sec_sim}). We rely heavily on the analyticity of $G$, which makes many computations more tractable.

\subsection{Organization}
The rest of the article is organized as follows. Section \ref{sec_prelim} fixes the notation and provides background on the GEF and its derivative. Section \ref{sec_appkr} analyzes the so-called Kac-Rice formulas for the second moments of the statistics \eqref{eq_stats} and exploits various symmetries of GEF to provide simplified expressions. 
Section \ref{sec_apex} lays the ground to analyze such expressions, by providing asymptotic expansions for the various quantities involved. 
Each of the sections \ref{sec_cp1}, \ref{sec_cps}, \ref{sec_cns}, \ref{sec_zp1} and \ref{sec_zp2} contains a proof of one of the claims of Theorem \ref{th_main}. A full proof of Theorem \ref{th_main} is then provided in Section \ref{sec_proof}. Finally, some technical results and computations are postponed to the Appendix (Section \ref{sec_app}). We also provide an online worksheet to follow some of the calculations with symbolic software \cite{hor25work}.

\section{Preliminaries}\label{sec_prelim}
The real and imaginary parts of $z \in \mathbb{C}$ are denoted $\mathrm{Re}(z)$ and $\mathrm{Im}(z)$, respectively. The differential of the (Lebesgue) area measure on the plane will be denoted $dA$, while the Lebesgue measure of a set $E$ is $|E|$. The indicator function of a set $E$ is $\mathrm{1}_E$.

We say that a random variable $X$ has \emph{stretched exponential tails} if there exist constants $K,k,\gamma>0$ such that
\begin{align}
\mathbb{P} \big[ |X|>t \big] \leq K e^{-kt^\gamma}, \qquad t>0.
\end{align}
In concrete applications we will be interested in the uniformity of these constants on other parameters.
Linear combinations and products of random variables with stretched exponential tails have stretched exponential tails, though the corresponding constants may change.

The conditioning of a normal vector $(X,Y) \in \mathbb{R}^n \times \mathbb{R}^m$ to $Y=0$ is defined by \emph{Gaussian regression} --- see, e. g., \cite[Chapter 1]{level}. Informally, this involves finding a linear combination of $X$ and $Y$ which is uncorrelated to $Y$.

By a complex normal vector $Z$, we mean a circularly symmetric complex Gaussian random vector, i.e., a random vector $Z$ on $\mathbb{C}^n$ such that
$(\mathrm{Re} (Z), \mathrm{Im} (Z))$ is normally distributed, has zero mean, and vanishing \emph{pseudo-covariance}: $\mathbb{E}\big[ Z   Z^t \big] = 0$.
A complex normal vector $Z$ on $\mathbb{C}^n$ is thus determined by its \emph{covariance matrix}
\begin{align*}
\cov [Z] = \mathbb{E} \big[ Z   Z^* \big],
\end{align*}
and we write $Z \sim \mathcal{N}_{\mathbb{C}}(\Sigma)$ to indicate $Z$ is complex normal with covariance matrix $\Sigma$.
Normal vectors are not a priori assumed to have non-singular covariances. The zero vector, for example, is a singular normal vector.

The derivatives of a function $f\colon \mathbb{C} \to \mathbb{C}$ interpreted as $f\colon \mathbb{R}^2 \to \mathbb{R}^2$ are denoted by $f^{(1,0)}$ (real coordinate) and $f^{(0,1)}$ (imaginary coordinate). Higher derivatives are denoted by $f^{(k,\ell)}$. We make extensive use of the \emph{Wirtinger operators}
\begin{align}
\partial f = \frac{1}{2} \big( f^{(1,0)} - i f^{(0,1)} \big),
\\
\overline{\partial} f = \frac{1}{2} \big( f^{(1,0)} + i f^{(0,1)} \big),
\end{align}
and the adjoint of $\overline{\partial}$ with respect to $L^2$ of the Gaussian weight:
\begin{align}\label{eq_dbarstar}
\overline{\partial}^* f(z) = \bar{z} f(z) - \partial f(z),
\end{align}
also known as the \emph{covariant derivative}. The \emph{Jacobian} of $f\colon \mathbb{C} \to \mathbb{C}$ at $z \in \mathbb{C}$ is the determinant of its differential matrix $Df$, considering $f$ as $f\colon \mathbb{R}^2 \to \mathbb{R}^2$:
\begin{align*}
\jac f(z) := \det Df(z).
\end{align*}
The following observations will be used repeatedly:
\begin{align}\label{eq_Jim}
\jac f(z) = - \mathrm{Im}\big[  f^{(1,0)}(z) \cdot \overline{ f^{(0,1)}(z)}\big]
= |\partial f(z)|^2-|\overline{\partial} f (z)|^2.
\end{align}

We will always let $G$ be the Gaussian entire function \eqref{eq_gef} and $F$ its covariant derivative \eqref{eq_F}. Inspecting the random power series \eqref{eq_gef} we see that, for every $k \in \mathbb{N}_0$,
\begin{align}
\big(G(0), \partial G(0), \ldots, \tfrac{1}{\sqrt{k!}} \partial^{k}G(0)\big)
\end{align}
is a standard (circularly symmetric) complex normal vector (identity covariance matrix).
Other correlations of the Gaussian entire function \eqref{eq_gef} can be obtained from its
\emph{covariance kernel}:
\begin{align}\label{eq_kergef}
\mathbb{E}\big[G(z) \cdot \overline{G(w)}\big] = e^{z \overline{w}}, \qquad z, w \in \mathbb{C}.
\end{align}
For example, correlations between derivatives of $G$ can be computed by exchanging differentiation and expectation, as we do in Lemma \ref{matrixcomp}. In particular,
\begin{align}\label{eq_cF}
\mathbb{E}\big[F(z) \cdot \overline{F(w)}\big] &= (1-|z-w|^2) e^{z \overline{w}}, \qquad z, w \in \mathbb{C},
\\\label{eq_cF_2}
\mathbb{E}\big[F(z) \cdot \overline{G(w)}\big] &= \overline{(z-w)} e^{z \overline{w}}, \qquad z, w \in \mathbb{C}.
\end{align}
Equation \eqref{eq_cF} shows that for $z\not= w$, the vector $(F(z),F(w))$ is non-degenerate, see also \cite[Example 1.3 and Section 6]{hkr22}.

\section{Approximate Kac-Rice formulas}\label{sec_appkr}
\subsection{Symmetries of GEF}\label{sec_sim}
The stochastics of the Gaussian entire function $G$ are invariant under the \emph{Bargmann-Fock shifts}
\begin{align}
f(z) \mapsto f_\zeta(z) = e^{-\frac{|\zeta|^2}{2} + z \overline{\zeta}} f(z- \zeta), \qquad \zeta \in \mathbb{C},
\end{align}
as can be seen by considering their effect on the covariance kernel \eqref{eq_kergef}:
\begin{align}
\mathbb{E} \big[ G_\zeta(z) \cdot \overline{G_\zeta(w)} \big] = 
e^{-|\zeta|^2 + z \overline{\zeta} +\overline{w} \zeta} e^{(z-\zeta)\overline{(w-\zeta)}} = e^{z \overline{w}} = \mathbb{E} \big[ G(z) \cdot \overline{G(w)} \big].
\end{align}
In addition, the covariant derivative \eqref{eq_dbarstar} commutes with the Bargmann-Fock shifts:
\begin{align}
\overline{\partial}^* [f_\zeta(z)] = [\overline{\partial}^* f]_\zeta(z).
\end{align}
As a consequence, the Gaussian field $(G,F)$ is stochastically invariant under the double shifts:
\begin{align}\label{eq_zeta}
(G_\zeta,F_\zeta) \stackrel{(d)}{=} (G,F).
\end{align}
In Lemma \ref{lemmashift} in the appendix, we present some consequences of this for the moments of the statistics \eqref{eq_stats}. In fact, it follows from \eqref{eq_zeta} that \eqref{eq_stats} can also be calculated with respect to a disk $B_\rho(\zeta)$ with center $\zeta \in \mathbb{C}$ without affecting their stochastics. Importantly, though $G$ is invariant under rotations --- $G(z) \sim G(\lambda z)$, $|\lambda|=1$ --- the pair $(F,G)$ does not satisfy a similar invariance, which makes the analysis of correlations between zeros and critical points more challenging.

\subsection{Approximate intensities}
The so-called Kac-Rice formulas give intensities for the first and second moments of the statistics \eqref{eq_stats} --- see, e. g., \cite[Theorem 6.2] {level}, \cite[Chapter 11]{adler}. We now write approximate forms of some of these.
\begin{prop}[Approximate Kac-Rice formulas] \label{approxKacRice}
For $r>0$, let
\begin{align}\label{eq_cr}
\sigma^{c}(r) := \mathbb{E}\big[|\jac F(ir)\cdot\jac F(-ir)| \,\big|\, F(ir) = F(-ir)= 0\big],
\end{align}
\begin{align}\label{eq_crp}
\sigma^{c,+}(r) := \mathbb{E}\big[|\jac F(ir)\cdot\jac F(-ir)| \cdot \mathrm{1}_{\jac F(ir) > 0} \cdot \mathrm{1}_{\jac F(-ir) > 0} \,\big|\, F(ir) = F(-ir)= 0\big],
\end{align}
\begin{align}\label{eq_crm}
\sigma^{c,-}(r) := \mathbb{E}\big[|\jac F(ir)\cdot\jac F(-ir)| \cdot \mathrm{1}_{\jac F(ir) < 0} \cdot \mathrm{1}_{\jac F(-ir) < 0} \,\big|\, F(ir) = F(-ir)= 0\big].
\end{align}
 Let $(\mathcal{N}_\rho, \sigma)$ be given by $(\mathcal{N}_\rho^c, \sigma^{c}), (\Ncp, \sigma^{c,+})$ or $(\Ncm, \sigma^{c, -})$. Then
\begin{align}\label{eq_achan}
    &\rho^2 \int_0^{\frac{\rho}{4}} \sigma(r) r^{-1} dr \lesssim \mathbb{E}[\mathcal{N}_\rho \cdot (\mathcal{N}_\rho - 1)] \lesssim \rho^2 \int_0^{\rho} \sigma(r) r^{-1} dr, \qquad 0<\rho \leq 1.
\end{align}
\end{prop}
\begin{proof}
The Kac-Rice formulas give the following expressions for the correlations of zeros and critical points of various indices, valid for $0<\rho \leq 1$:
\begin{align}\label{kr1}    
        &\mathbb{E}[\Nc (\Nc - 1)] = \int_{B_\rho^2} \frac{1}{\pi^2} \frac{\mathbb{E}\big[|\jac F(z) \jac F(w)| \,\big|\, F(z) = F(w)= 0\big]}{e^{|z|^2 + |w|^2} \big(1- e^{-|z-w|^2}(1 - |z-w|^2)^2\big) } dA(z)dA(w),
\end{align}
\begin{align}\label{kr2}
        &\mathbb{E}[\Ncp (\Ncp - 1)] \\
        &= \int_{B_\rho^2} \frac{1}{\pi^2} \frac{\mathbb{E}\big[|\jac F(z) \jac F(w)|\mathrm{1}_{\jac F(z) > 0} \mathrm{1}_{\jac F(w) > 0} \big| F(z) = F(w)= 0\big]}{e^{|z|^2 + |w|^2} \big(1- e^{-|z-w|^2}(1 - |z-w|^2)^2\big)  } dA(z)dA(w),
\end{align}
\begin{align} \label{kr3}
        &\mathbb{E}[\Ncm (\Ncm - 1)] \\
        &= \int_{B_\rho^2} \frac{1}{\pi^2} \frac{\mathbb{E}\big[|\jac F(z) \jac F(w)|\mathrm{1}_{\jac F(z) < 0} \mathrm{1}_{\jac F(w) < 0} \big| F(z) = F(w)= 0\big] }{e^{|z|^2 + |w|^2} \big(1- e^{-|z-w|^2}(1 - |z-w|^2)^2\big) } dA(z)dA(w),    
\end{align}
see Lemma \ref{kacrice} in the appendix for details. As it happens, the fundamental symmetry \eqref{eq_zeta} implies that the integrand in each of the expressions above depends only on $|z-w|$; we provide the details in Lemma \ref{lemmashift} (appendix). We use this fact to replace the expressions depending on $(z,w)$ by the same expressions evaluated at $(ir,-ir)$, where $r=|z-w|/2$. For $(\mathcal{N}_\rho, \sigma)$ given by $(\mathcal{N}_\rho^c, \sigma^{c}), (\Ncp, \sigma^{c,+})$ or $(\Ncm, \sigma^{c, -})$ this yields:
\begin{align}\label{eq_kr_exact}
    \mathbb{E}[\mathcal{N}_\rho \cdot(\mathcal{N}_\rho - 1)] = \int_{B_\rho^2} \frac{1}{\pi^2}\frac{\sigma({|z-w|}/{2})}{e^{\frac{|z-w|^2}{2}} \big(1- e^{-|z-w|^2}(1 - |z-w|^2)^2\big) }  dA(z) dA(w).
\end{align}
Noting that, for $|z|,|w|<\rho\leq 1$,
\begin{align}
    {e^{\frac{|z-w|^2}{2}} \cdot \big(1- e^{-|z-w|^2}(1 - |z-w|^2)^2\big) } \asymp |z-w|^{2},
\end{align}
we conclude that
\begin{align}
    \mathbb{E}[\mathcal{N}_\rho \cdot (\mathcal{N}_\rho - 1)]  \asymp \int_{B_\rho^2} |z-w|^{-2} \sigma({|z-w|}/{2}) dA(z) dA(w).
\end{align}
Finally, this area integral can be approximately reduced to a one variable integral over $r=|z-w|/2$ (see Lemma \ref{lemmared} in the appendix) giving 
\begin{align*}
    \rho^2 \int_0^{\frac{\rho}{2}} \sigma(r/2) r^{-1} dr 
    \lesssim \mathbb{E}[\mathcal{N}_\rho \cdot (\mathcal{N}_\rho - 1)] \lesssim \rho^2 \int_0^{2\rho} \sigma(r/2) r^{-1} dr.
\end{align*}
A change of variables gives \eqref{eq_achan}.  
\end{proof}

\section{Asymptotic expansions}\label{sec_apex}
\subsection{The Jacobian of F}
Motivated by the approximate intensities provided by Proposition \ref{approxKacRice}, we look into approximately expanding the Jacobian of $F$. Let us introduce the following \emph{proxy random variables}:
\begin{align}
&A = \mathrm{Im}\big(F^{(0, 2)}(0) \cdot \overline{F^{(1, 0)}(0)}\big),
\\
&B = \mathrm{Im}\big(i |F^{(0, 2)}(0)|^2 + \tfrac{1}{3} F^{(0, 3)}(0)\cdot\overline{F^{(1, 0)}(0)}\big),
\end{align}
and retain this notation throughout the remainder of the article. The next proposition provides a suitable  asymptotic description of the Jacobian of $F$ in terms of the proxies $A$ and $B$.
\begin{prop} \label{condjac}
For $r \in (0,1)$ there exist random variables $C_r^+, C_r^-, D_r$ and $F_r$ with the following properties:
\begin{itemize}
\item[(i)] \emph{(Expansions)}:
\begin{align}
        &\jac F(ir) = \mathrm{Im}(F^{(0, 1)}(ir) \overline{F^{(1, 0)}(ir)}) = r A + r^2 B + r^3 C_r^+,
    \\
        &\jac F(-ir) = \mathrm{Im}(F^{(0, 1)}(-ir) \overline{F^{(1, 0)}(ir)}) = -r A + r^2 B + r^3 C_r^-,
        \\
        &\jac F(ir)\jac F(-ir) = r^2(-A^2 + r^2 B^2 + r^2 A D_r + r^3 F_r).
    \end{align}
\item[(ii)] \emph{(Error bounds)}: When conditioned on $F(ir) = F(-ir) = 0$, the variables $A$, $B$, $C_r^+$, $C_r^-$, $D_r$, $F_r$ have stretched exponential tails with constants independent of $r \in (0,1)$. More precisely,
there exist constants $k,K,\gamma>0$ such that
\begin{align*}
\mathbb{P}\big[ |X| > t\,|\,F(ir) = F(-ir) = 0\big] \leq K e^{-k t^\gamma}, \qquad \mbox{for all }t>0, r \in (0,1),
\end{align*}
and $X \in \{A, B, C_r^+, C_r^-, D_r, F_r\}$.
\end{itemize}
\end{prop}
\begin{proof}
We let $r>0$ and expand $F(ir), F(-ir)$ around $0$ as
    \begin{equation*}
        F(ir) = F(0) + r F^{(0, 1)}(0) + \frac{r^2}{2} F^{(0, 2)}(0) + \frac{r^3}{6} F^{(0, 3)}(0) + r^4 E_1^{r},
    \end{equation*}
    \begin{equation*}
        F(-ir) = F(0) - r F^{(0, 1)}(0) + \frac{r^2}{2} F^{(0, 2)}(0) - \frac{r^3}{6} F^{(0, 3)}(0) + r^4 E_2^{r}.
    \end{equation*}
    Conditionally on 
    \begin{align}\label{eq_con}
    F(ir) = F(-ir) = 0
    \end{align}
    we get
    \begin{equation*}
        0 = F(ir) - F(-ir) = 2r F^{(0, 1)}(0) + 2 \frac{r^3}{6} F^{(0, 3)}(0) + r^4 E_1^{r} - r^4 E_2^{r},
    \end{equation*}
    and thus
    \begin{equation*}
        F^{(0, 1)}(0) = -\frac{r^2}{6} F^{(0, 3)}(0) + \frac{r^3}{2}(E_2^{r} - E_1^{r}).
    \end{equation*}
    Next, we expand first order derivatives and conclude that, conditionally on \eqref{eq_con},
    \begin{align}
            F^{(0, 1)}(ir) &= F^{(0, 1)}(0) + rF^{(0, 2)}(0) + \frac{r^2}{2} F^{(0, 3)}(0) + r^3 E_3^{r} \\\label{eq_x1}
            &= r F^{(0, 2)}(0) + \frac{r^2}{3} F^{(0, 3)}(0) + \frac{r^3}{2}(E_2^{r} - E_1^{r}) + r^3 E_3^{r},
  \\
            F^{(0, 1)}(-ir) &= F^{(0, 1)}(0) - rF^{(0, 2)}(0) + \frac{r^2}{2} F^{(0, 3)}(0) + r^3 E_4^{r} \\
            &= -r F^{(0, 2)}(0) + \frac{r^2}{3} F^{(0, 3)}(0) + \frac{r^3}{2}(E_2^{r} -E_1^{r}) + r^3 E_4^{r},
\\
        F^{(1, 0)}(ir) &= F^{(1, 0)}(0) + r F^{(1, 1)}(0) +  r^2 E_5^{r},
 \\
        F^{(1, 0)}(-ir) &= F^{(1, 0)}(0) - r F^{(1, 1)}(0) + r^2 E_6^{r}.
    \end{align}
    Recalling that $F(z) = \overline{z} G(z) - \partial G(z)$ and that $G$ is analytic, the Cauchy-Riemann equations give
    \begin{equation*}
        F^{(0, 2)}(0) = -2i G^{(0,1)} (0) - (\partial G)^{(0,2)}(0)\\
        = 2\partial G(0) + \partial^3 G(0),
    \end{equation*}
    and
    \begin{equation*}
        F^{(1, 1)}(0)  
        =G^{(0,1)}(0)-iG^{(1,0)}(0)-(\partial G)^{(1,1)}(0)
        =- i \partial^3 G(0),
    \end{equation*}
    which implies
    \begin{align}\label{eq_a1}
    F^{(1, 1)}(0) = -i F^{(0, 2)}(0) + 2 i\partial G(0).
    \end{align}
    We also expand
    \begin{equation*}
        \partial G(ir) = \partial G(0) + r E_7^{r}
    \end{equation*}
    and use \eqref{eq_a1} to conclude that, conditionally on $F(ir) = -ir G(ir) - \partial G(ir) = 0$, 
    \begin{align}
        F^{(1, 1)}(0) &= -i F^{(0, 2)}(0) +  2i\partial G(0) = -i F^{(0, 2)}(0) + i (2\partial G(ir) - 2rE_7^{r}) \\
        &= -i F^{(0, 2)}(0) + i(-2i r G(ir)- 2r E_7^{r}) = -i F^{(0, 2)}(0) + 2 r G(ir) - 2 i r E_7^{r}.
    \end{align}
    Thus, conditionally on \eqref{eq_con}, we have
    \begin{align*}
        F^{(1, 0)}(ir) &= F^{(1, 0)}(0) - ir  F^{(0, 2)}(0) +  r^2 (E_5^{r} + 2 G(ir) - 2 i E_7^{r}),\\
        F^{(1, 0)}(-ir) &= F^{(1, 0)}(0) + ir F^{(0, 2)}(0) + r^2 (E_6^{r} - 2G(ir) +2 i E_7^{r}).
    \end{align*}
    Combining this with \eqref{eq_x1}, we conclude that, conditionally on \eqref{eq_con},
    \begin{align}
        &\jac F(ir) = \mathrm{Im}(F^{(0, 1)}(ir) \overline{F^{(1, 0)}(ir)}) = r A + r^2 B + r^3 C_r^+,\\
        &\jac F(-ir) = \mathrm{Im}(F^{(0, 1)}(-ir) \overline{F^{(1, 0)}(-ir)}) = -r A + r^2 B + r^3 C_r^-,\\
        &\jac F(ir)\cdot\jac F(-ir) = r^2(-A^2 + r^2 B^2 + r^2 A D_r + r^3 F_r),
    \end{align}
    where $C_r^+, C_r^-, D_r$ and $F_r$ are given by a finite sum of products of $r, E_i^{r}, G(ir), F^{(n, m)}(0)$ for $i = 1, ..., 7$ and $(n, m) \in \{(1, 0), (0, 2), (0, 3)\}$.

    Finally, we note that all the error factors have stretched exponential tails. To this end, consider the jointly Gaussian, circularly symmetric, zero mean random variables
    $G(z)$, $\partial^2 G(z)$, $F^{(n, m)}(z), 0 \leq n,m \leq 4$, and enumerate them as $X_1(z), \ldots, X_N(z)$. Let
    \begin{align}
        &E = 1 + \max_{1 \leq j \leq N}
        \sup_{z \in \overline{\mathbb{D}}} |(X_j(z)\, |\, F(ir) = F(-ir) = 0)|.
    \end{align}
    Then there exist $n \in \mathbb{N}$ and $C>0$ independent of $r \in (0,1)$ such that
    \begin{equation*}
        |(Y | F(ir) = F(-ir) = 0)|  \leq C E^n,
        \qquad 
        Y \in \{A,B, C_r^\pm, D_r, F_r\}.
    \end{equation*}
     Hence, it is enough to show $E$ has Gaussian tails, with parameters independent of $r \in (0,1)$. This is the case by general facts concerning Gaussian processes with smooth covariances. Specifically, for $j=1,\ldots,N$, since conditioning zero-mean jointly Gaussian variables reduces their variance, we have
    \begin{align}
    \sup_{z \in \overline{\mathbb{D}}} \var[X_j(z)\,|\,F(ir) = F(-ir) = 0] \leq
    \sup_{z \in \overline{\mathbb{D}}} \var[X_j(z)] \leq C,
    \end{align}
    for a constant independent of $r$. Second, the Gaussian field $Z_j(\cdot):=X_j(\cdot)\,|\,\big(F(ir) = F(-ir) = 0\big)$ --- defined by Gaussian regression, since $(F(ir), F(-ir))$ has an invertible covariance matrix, cf. Section \ref{sec_prelim} --- has zero mean and smooth covariance, so Dudley's and the Borell-TIS inequalities \cite[Theorem 2.9 and 2.10]{level} \cite[Theorems 1.3.3 and 2.1.1]{adler}, imply that $\sup_{z\in\overline{\mathbb{D}}} |Z_j(z)|$ has Gaussian tails with parameters independent of $r$. Of course, the same conclusion extends to $E$. (A similar argument was used in \cite[Proposition 1]{ladgham2023local}.)
\end{proof}

\begin{coro}\label{coro_rho4}
With the notation of Proposition \ref{approxKacRice}, let $(\mathcal{N}_\rho, \sigma)$ be given by $(\mathcal{N}_\rho^c, \sigma^{c}), (\Ncp, \sigma^{c,+})$ or $(\Ncm, \sigma^{c, -})$. Then
\begin{align}\label{eq_achan2}
    &\mathbb{E}[\mathcal{N}_\rho \cdot (\mathcal{N}_\rho - 1)] \lesssim \rho^4, \qquad 0<\rho < 1.
\end{align}
\end{coro}
\begin{proof}
We use the expansion and tail estimates in Proposition \ref{condjac} to bound, for $0<r < 1$,
\begin{align}
\sigma(r) 
&\lesssim \mathbb{E}\big[|\jac F(ir)\cdot\jac F(-ir)| \,\big|\, F(ir) = F(-ir)= 0\big]
\\
&\leq r^2 \cdot \mathbb{E}\big[|(-A^2 + r^2 B^2 + r^2 A D_r + r^3 F_r)| \,\big|\, F(ir) = F(-ir)= 0\big]
\lesssim r^2.
\end{align}
Hence, \eqref{eq_achan} gives 
$\mathbb{E}[\mathcal{N}_\rho \cdot (\mathcal{N}_\rho - 1)] \lesssim \rho^2 \int_0^{\rho} r dr \lesssim \rho^4$.
\end{proof}

\subsection{Asymptotic correlations among conditioned derivatives}
We continue by exploring the proxy variables $A$ and $B$ and look into the asymptotic correlations between the vectors that define them.
\begin{lemma}\label{conddens}
The covariance matrix of $(F^{(1, 0)}(0), F^{(0, 2)}(0), F^{(0, 3)}(0))$ conditioned on $F(ir) = F(-ir) = 0$ converges, as $r \to 0^+$, to  
\begin{align}\label{eq_M0}
M^0 := \begin{pmatrix}
\frac{8}{3} & 0 & 4 i \\ 
0 & 6 & 0 \\
-4 i & 0 & 30
\end{pmatrix}.
\end{align}
As a consequence:
\begin{itemize}
\item[(a)] If $h: \mathbb{C}^3 \to \mathbb{C}$ is measurable and has at most polynomial growth, then
\begin{align}\label{eq_opopa}
\lim_{r\to0^+} \mathbb{E} \big[h(F^{(1, 0)}(0), F^{(0, 2)}(0), F^{(0, 3)}(0))\,\big|\, F(ir) = F(-ir) = 0 \big] =
\mathbb{E} \big[h(Z_1, Z_2, Z_3)\big],
\end{align}
where $(Z_1,Z_2,Z_3) \sim \mathcal{N}_{\mathbb{C}}(M^0)$.

\item[(b)] There exist $\alpha,\beta,\delta,c,C>0$ such that, if $h:\mathbb{C}^3 \to [0,+\infty)$ is measurable, then, for $0<r<\delta$:
\begin{align}\label{eq_opop}
c\cdot\mathbb{E} \big[h(\alpha Z_1, \alpha Z_2, \alpha Z_3)\big] \leq
\mathbb{E} \big[h(F^{(1, 0)}(0), F^{(0, 2)}(0), F^{(0, 3)}(0))\,\big|\, F(ir) = F(-ir) = 0\big]
\\
\leq C\cdot \mathbb{E} \big[h(\beta Z_1, \beta Z_2, \beta Z_3)\big],
\end{align}
where $(Z_1,Z_2,Z_3) \sim \mathcal{N}_{\mathbb{C}}(I)$.

\item[(c)] In particular, in the situation of $(b)$, if $h$ is $k$-homogenoues, $k \in \mathbb{N}$, then
\begin{align}\label{eq_opopb}
\mathbb{E} \big[h(F^{(1, 0)}(0), F^{(0, 2)}(0), F^{(0, 3)}(0))\,\big|\, F(ir) = F(-ir) = 0\big]
\asymp \mathbb{E} \big[h(Z_1, Z_2, Z_3)\big],
\end{align}
where $(Z_1,Z_2,Z_3) \sim \mathcal{N}_{\mathbb{C}}(I)$ and the implied constants depend on $k$.
\end{itemize}
\end{lemma}
\begin{proof}
    A direct computation shows that the covariance matrix of the vector \[(F(ir), F(-ir), F^{(1, 0)}(0), F^{(0, 2)}(0), F^{(0, 3)}(0))\] is
    \begin{equation}\label{eq_mtotr}
        \begin{pmatrix}
            M_1^r & M_2^r \\
            {(M_2^r)^*} & M_3^r
        \end{pmatrix},
    \end{equation}
    where
    \begin{equation}\label{eq_m1r}
        M_1^r =  \begin{pmatrix}
            e^{r^2} & e^{-r^2}(1 - 4r^2) \\
            e^{-r^2}(1 - 4r^2) & e^{r^2}
        \end{pmatrix},
    \end{equation}
    \begin{equation}\label{eq_m2r}
        M_2^r = \begin{pmatrix}
            ir(1-r^2) & -2+5r^2 - r^4 & -6r + 7r^3 - r^5 \\
            - i r( 1-r^2) & -2 + 5r^2 -r^4 & 6 r - 7r^3 + r^5 
        \end{pmatrix},
    \end{equation}
    \begin{equation}\label{eq_m3r}
        M_3^r = \begin{pmatrix}
            3 & 0 & 6i \\
            0 & 10 & 0 \\
            -6i & 0 & 42& 
        \end{pmatrix},
    \end{equation}
    see Lemma \ref{matrixcomp} in the appendix or \cite{hor25work}. Let $X^r$ denote the conditioned vector
 \[X^r := \big(\,(F^{(1, 0)}(0), F^{(0, 2)}(0), F^{(0, 3)}(0))\,\big|\,F(ir) = F(-ir) = 0\,\big),\] defined by Gaussian regression. Its covariance is
    \begin{align}
        M^r &= M_3^r - {(M_2^r)^*} (M_1^r)^{-1} M_2^r \\
        &= \begin{pmatrix} 3 - \frac{2 e^{r^2} r^2 (1-r^2)^2}{-1 + e^{2r^2} + 4r^2} & 0 & 6i + 2i \frac{e^{r^2}r^2 (-6+ r^2)(r^2 - 1)^2}{-1+e^{2r^2}+4r^2}\\
        0 & 10 - \frac{2e^{r^2}(2-5r^2+r^4)^2}{1+e^{2r^2}-4r^2} & 0 \\
        -6i -2i\frac{e^{r^2}r^2 (-6+ r^2)(r^2 - 1)^2}{-1+e^{2r^2}+4r^2} & 0 & 42 - \frac{2e^{r^2}r^2(6-7r^2+r^4)^2}{-1+e^{2r^2}+4r^2} \end{pmatrix},
    \end{align}
    which indeed converges to $M^0$ as $r\to0^+$ (this can be verified directly or with \cite{hor25work}). Since $M^0$ is positive definite, there exist $\delta>0$ and constants $\tau_1, \tau_2>0$ such that (in the Loewner order),
    \begin{align}
    \tau_1 I \leq M^r \leq \tau_2 I, \qquad 0<r<\delta.
    \end{align}
    It follows that $X^r$ is an absolutely continuous random variable for $0<r<\delta$; let $f^r$ denote its probability density. Moreover, if $f$ denotes the probability density of a standard complex vector on $\mathbb{C}^3$, we also have
    \begin{align}\label{eq_kljlkl}
    c f(z/\alpha) \leq f^r(z) \leq C f(z/\beta), \qquad z=(z_1,z_2,z_3) \in \mathbb{C}^3,
    \end{align}
      where $\alpha,\beta,c,C>0$ are absolute constants.

Part (b) now follows easily: if $Z=(Z_1,Z_2,Z_3) \sim \mathcal{N}_{\mathbb{C}}(I)$, then
    \begin{align}\label{eq_kljlkl2}
    &\mathbb{E} \big[h(F^{(1, 0)}(0), F^{(0, 2)}(0), F^{(0, 3)}(0))\,\big|\, F(ir) = F(-ir) = 0\big]
    \\
    &\quad = \int_{\mathbb{C}^3} h(z_1,z_2,z_3) f^r(z_1,z_2,z_3) \,dA(z_1)\,dA(z_2)\,dA(z_3)
    \end{align}
    and we use \eqref{eq_kljlkl} to estimate
    \begin{align}
    &\mathbb{E} \big[h(F^{(1, 0)}(0), F^{(0, 2)}(0), F^{(0, 3)}(0))\,\big|\, F(ir) = F(-ir) = 0]
    \\
    &\quad \leq C
    \int_{\mathbb{C}^3} h(z_1,z_2,z_3) f(z_1/\beta,z_2/\beta,z_3/\beta) \,dA(z_1)\,dA(z_2)\,dA(z_3)
    \\
    &\quad = C \beta^6 \int_{\mathbb{C}^3} h(\beta z_1, \beta z_2, \beta z_3) f(z_1,z_2,z_3) \,dA(z_1)\,dA(z_2)\,dA(z_3)
    \\
    &\quad = C \beta^6 \,\mathbb{E} \big[h(\beta Z_1, \beta Z_2, \beta Z_3)\big].
    \end{align}
    The lower bound in \eqref{eq_opop} follows similarly. Part (c) follows immediately from part (b).
    As for part (a), if $Z=(Z_1,Z_2,Z_3) \sim \mathcal{N}_{\mathbb{C}}(M^0)$, then $f^r \to f_{Z}$ almost everywhere as $r \to 0^+$. We use \eqref{eq_kljlkl}, together with the fact that $h$ grows at most polynomially, to exchange the limit with integration in 
    \eqref{eq_kljlkl2}, which yields \eqref{eq_opopa}.
    \end{proof}

\section{Critical points}\label{sec_cp1}
As a warm-up, we derive the local correlations among the critical points of GEF. The following proposition is essentially the main of result \cite{MR2966361} (though the setting there is slightly different).
\begin{prop} \label{repc}
    There exist constants $c, C > 0$ such that
    \begin{align}\label{eq_repca}
        c \rho^4 \leq \mathbb{E}[\Nc \cdot (\Nc - 1)] \leq C \rho^4, \qquad 0 < \rho < 1.
    \end{align}
In addition, 
\begin{align}\label{eq_repcb}
\lim_{\rho \to 0^+} \frac{\mathbb{E}[\Nc \cdot (\Nc - 1)]}{(\mathbb{E}[\Nc])^2} = \frac{6}{25}.
\end{align}
 \end{prop}
\begin{proof}
The upper bound in \eqref{eq_repca} follows from Corollary \ref{coro_rho4}. As $\Nc\cdot(\Nc - 1)$ is an increasing function of $\rho$, in order to prove the lower bound in \eqref{eq_repca} it is enough to consider sufficiently small but positive $\rho$. The first order moment is
$\mathbb{E}[\Nc] = \tfrac{5}{3}  \rho^2$,
as follows from a direct calculation with Kac-Rice's formula, see \cite[Corollary 1.7]{hkr22}. Hence, to complete the proof of \eqref{eq_repca} and \eqref{eq_repcb}, we must show that
\begin{align}
\lim_{\rho \to 0^+} \frac{\mathbb{E}[\Nc \cdot (\Nc - 1)]}{\rho^4} = \frac{2}{3}.
\end{align}
We use the exact version of the Kac-Rice formula \eqref{eq_kr_exact}, with $\sigma=\sigma^{c}$ given by \eqref{eq_cr},
    which by Proposition \ref{condjac} satisfies
    \begin{align}
        \sigma^{c}(r) = r^2 \cdot \mathbb{E}\big[A^2 \big|F(ir) = F(-ir) = 0\big] + O(r^4), \qquad 0 < r< 1,
    \end{align}
    as
    $
    \mathbb{E}\big[|B^2 + A D_r + r F_r| \,|\, F(ir) = F(-ir) = 0\big] = \Ordo(1)$, for $0 < r \leq 1$.
    Recalling the definition of $A$, by Lemma \ref{conddens}, if $Z_1, Z_2$ are independent complex variables
    with variances $8/3$ and $6$ respectively, we have
    \begin{align}
    \mathbb{E}\big[A^2 \big|F(ir) = F(-ir) = 0\big] \sim 
    \mathbb{E}\big[\big(\mathrm{Im} (Z_2 \overline{Z_1}) \big)^2\big] = 
    \frac{1}{2} \frac{8}{3} 6 = 8, \qquad \mbox{as } r\to0^+.
    \end{align}
    Hence, $\sigma^c(r) \sim 8r^2$ and
  \begin{align}
  &\lim_{\rho \to 0^+} \frac{1}{\rho^4} \mathbb{E}[\Nc\cdot(\Nc - 1)]
  = \lim_{z,w \to 0}
  \frac{\sigma^c({|z-w|}/{2})}{e^{\frac{|z-w|^2}{2}} \big(1- e^{-|z-w|^2}(1 - |z-w|^2)^2\big) }
  \\
  &\qquad=\lim_{r \to 0^+}
  \frac{\sigma^c(r)}{e^{{2r^2}} \big(1- e^{-4r^2}(1 -4r^2)^2\big) }
  =\lim_{r \to 0^+}
  \frac{8 r^2}{e^{{2r^2}} \big(1- e^{-4r^2}(1 -4r^2)^2\big) } = \frac{2}{3}.\qedhere
  \end{align} 
\end{proof}
Note that Proposition \ref{repc} only uses the first order of the asymptotic expansions in Proposition \ref{condjac} and only a subset of the correlations provided by Lemma \ref{conddens}. The more refined information provided by Proposition \ref{condjac} and Lemma 
\ref{conddens} will be required to study critical points with a given index.

\section{Critical points with positive index}\label{sec_cps}
We delve into the correlation among critical points of GEFs and start with those with positive index.
As a first step, we study the approximate intensity given by Proposition \ref{approxKacRice}.
\begin{lemma} \label{approxp}
    Let $\sigma^{c, +}(r)$ be the approximate intensity \eqref{eq_crp} and $0 <\alpha < 1$. Then
    \begin{equation*}
        \frac{\sigma^{c, +}(r)}{r^2} = \mathbb{E}\big[(-A^2 + r^2 B^2)\mathrm{1}_{|A| < r B} \,\big |\, F(ir) = F(-ir) = 0\big] + O(r^{3+\alpha}),
        \qquad 0<r<1,
    \end{equation*}
\end{lemma}
where the implied constant depends on $\alpha$.
\begin{proof}
With the notation of Proposition, \ref{condjac}, we have
    \begin{align*}
        \frac{\sigma^{c, +}(r)}{r^2} = \mathbb{E}\big[(-A^2 + r^2 B^2) \mathrm{1}_{\jac F(ir) > 0} \mathrm{1}_{\jac F(-ir) > 0} \,\big |\, F(ir) = F(-ir) = 0 \big] + b_r,
    \end{align*}
    where
    \begin{align*}
        b_r = \mathbb{E}\big[(r^2 A D_r + r^3 F_r) \mathrm{1}_{\jac F(ir) > 0} \mathrm{1}_{\jac F(-ir) > 0} \,\big |\, F(ir) = F(-ir) = 0 \big].
    \end{align*}
    Proposition \ref{condjac} implies that $|b_r| \lesssim 1$ for $0<r<1$, so we can focus on the small $r$ range. We fix
    $0<\alpha<1$ and let constants depend on it.

    \medskip \noindent {\bf Step 1}. We first show that 
    \begin{align}\label{eq_z6}
    |b_r| \lesssim r^{3+\alpha}.
    \end{align}
Inspecting the expansion of $\jac F(\pm ir)$ given by Proposition \ref{condjac}, we see that $\jac F(ir) > 0$ and $\jac F(-ir) > 0$ 
    occur simultaneously if and only if
    \begin{align*}
    A> -(rB+r^2C_r^+) \mbox{ and }
    -A> -(rB+r^2C_r^-).
    \end{align*}
    In this case
    $-A < rB + r^2 C_r^+$ and $
    A < rB+r^2C_r^-$, which gives
    \begin{equation*}
        |A| < r |B| + r^2 |C_r^+| + r^2 |C_r^-|.
    \end{equation*}
    Thus,
    \begin{align*}
        |b_r| \leq r^3 \mathbb{E}\big[(|D_r|(|B| + r |C_r^+| + r |C_r^-|) + |F_r| ) \mathrm{1}_{|A| < r(|B| + r |C_r^+| + r |C_r^-|)} \big| F(ir) = F(-ir) = 0\big].
    \end{align*}   
    By Hölder's inequality, with $p \in (1,\infty)$ to be determined, we further estimate
    \begin{align*}
    r^{-3}|b_r| &\leq \mathbb{E}\big[(|D_r|(|B| + r |C_r^+| + r |C_r^-|) + |F_r|)^{p'} \big| F(ir) = F(-ir) = 0\big]^{\frac{1}{p'}} \\
        &\qquad\cdot \mathbb{E}\big[\mathrm{1}_{|A| < r(|B| + r|C_r^+| + r C_r^-|)} \big| F(ir) =F(-ir) = 0\big]^{\frac{1}{p}}.
    \end{align*}
    We note that, for any $0 < \eta < 1$,
    \begin{align*}
        &\mathbb{E}\big[\mathrm{1}_{|A| < r(|B| + r|C_r^+| + r |C_r^-|)} \big| F(ir) =F(-ir) = 0\big] 
        \\
        &\qquad\leq \mathbb{E}\big[\mathrm{1}_{|A| < r^{1-\eta}} \big| F(ir) = F(-ir) = 0\big]
        + \mathbb{E}\big[\mathrm{1}_{|B| + r |C_r^+| + r |C_r^-| > r^{-\eta}} \big| F(ir) = F(-ir) = 0\big].
    \end{align*}
    By Proposition \ref{condjac}, conditioned on the event in question,    
    the random variable $|B| + r |C_r^+| + r |C_r^-|$ has stretched exponential tails with constants independent of $r \in (0,1)$. Hence, for constants $k,K, \gamma >0$,
    \begin{align*}
      &\mathbb{E}\big[\mathrm{1}_{|B| + r |C_r^+| + r |C_r^-| > r^{-\eta}} \,\big|\, F(ir) = F(-ir) = 0\big] \\
        &\qquad= \mathbb{P}(|B| + r |C_r^+| + r |C_r^-| > r^{-\eta} \big| F(ir) = F(-ir) = 0) \leq
        K e^{-k r^{-\eta \gamma}}
        \lesssim r^2, 
    \end{align*}
   where the implied constant depends on $\eta$ and
    \begin{equation*}
        \mathbb{E}\big[(|D_r|(|B| + r |C_r^+| + r |C_r^-|) + |F_r|)^{p'} \,\big|\, F(ir) = F(-ir) = 0\big]^{\frac{1}{p'}} \lesssim 1,
    \end{equation*}
     where the implied constant depends on $p$.

    In addition, by Lemma \ref{conddens}, for sufficiently small $r>0$,
     \begin{align*}
        &\mathbb{E}\big[\mathrm{1}_{|A| < r^{1-\eta}} \big| F(ir) = F(-ir) = 0\big] 
        \\
        &\qquad= \mathbb{E}[\mathrm{1}_{|\mathrm{Im}(F^{(0, 2)}(0) \overline{F^{(1, 0)}(0)})| < r^{1-\eta}} \big| F(ir) =F(-ir) = 0] \lesssim
        \mathbb{E}[\mathrm{1}_{|\mathrm{Im}(Z_2 \overline{Z_1})|< r^{1-\eta}}], 
    \end{align*}
    where $(Z_1,Z_2)$ is a multiple of a standard complex random vector (and the multiplying constant is absolute). A direct computation shows that
    \[
    \mathbb{P}(|\mathrm{Im}(Z_2 \overline{Z_1})|< r^{1-\eta}) \lesssim
    r^{1-\eta}(1+ |\log(r^{1-\eta})|),
    \]
    see Lemma \ref{lemappxdirect} below.

Hence, we conclude that, for sufficiently small $r > 0$
    \begin{align*}
        &|b_r| \lesssim r^3 (r^{1-\eta}(1+ |\log(r^{1-\eta})|) + r^2)^{\frac{1}{p}},
    \end{align*}
   where the implied constant depends on $p$ and $\eta$. We now fix $0<\alpha<1$, pick $\eta>0$ and $ p \in (1,\infty)$ such that $\alpha < (1-\eta)/p$ and $\alpha<2/p$, and let all subsequent constants depend on $\alpha$. With this understanding, for sufficiently small $r > 0$,
    \begin{equation*}
        |b_r| \lesssim r^3(r^\alpha + r^{2/p}) \lesssim r^{3+\alpha},
    \end{equation*}
    as desired.

\medskip\noindent {\bf Step 2}. We show that
\begin{align}\label{eq_z1}
e_r^+:=\mathbb{E}\big[\big((|B| + r |C_r^+| + r |C_r^-|)^2 +  B^2 \big) \mathrm{1}_{|A+ r B| < r^2 |C_r^+|} \,\big|\, F(ir) = F(-ir) = 0\big] \lesssim r^{1+\alpha},
\\\label{eq_z2}
e_r^-:=\mathbb{E}\big[\big((|B| + r |C_r^+| + r |C_r^-|)^2 +  B^2 \big) \mathrm{1}_{|-A+ r B| < r^2 |C_r^-|} \,\big|\, F(ir) = F(-ir) = 0\big] \lesssim r^{1+\alpha}.
\end{align}
As before, by Proposition \ref{condjac}, $|e^\pm_r| \lesssim 1$ so we can focus on small $r$. We start with \eqref{eq_z1}. Let $0<\eta<1$ and $p \in (1,\infty)$ be numbers to be specified as functions of $\alpha$, and estimate
    \begin{align}\label{eq_y1}
    \begin{aligned}
        &e^+_r\leq \mathbb{E}\big[\big((|B| + r |C_r^+| + r |C_r^-|)^2 +  B^2 \big)^{p'}  \big | F(ir) =F(-ir) = 0\big]^{\frac{1}{p'}}\\
        &\qquad\qquad\cdot \mathbb{E}\big[\mathrm{1}_{|A + rB| < r^2 |C_r^+|}\big |F(ir) = F(-ir) = 0\big]^{\frac{1}{p}},
    \end{aligned}
    \end{align}
and
    \begin{align}\label{eq_y2}
    \begin{aligned}
    &\mathbb{E}\big[\mathrm{1}_{|A + rB| < r^2 |C_r^+|}\big |F(ir) = F(-ir) = 0\big] 
    \\
    &\qquad\leq \mathbb{E}\big[\mathrm{1}_{|A + rB| < r^{2-\eta}}\big| F(ir) = F(-ir) = 0\big] + \mathbb{E}\big[\mathrm{1}_{|C_r^+| > r^{-\eta}} \big| F(ir) = F(-ir ) = 0\big].
    \end{aligned}
    \end{align}
    By the (conditional) tail estimates in Proposition \ref{condjac}, 
    \begin{equation}\label{eq_y3}
        \mathbb{E}\big[\mathrm{1}_{|C_r^+| > r^{-\eta}} \big| F(ir) = F(-ir ) = 0\big] = \mathbb{P}(|C_r^+| > r^{-\eta} \big| F(ir) = F(-ir) = 0) \lesssim r^2
    \end{equation}
    and
    \begin{equation}\label{eq_y4}
        \mathbb{E}\big[\big((|B| + r |C_r^+| + r |C_r^-|)^2 +  B^2 \big)^{p'}  \big | F(ir) =F(-ir) = 0\big]^{\frac{1}{p'}} \lesssim 1,
    \end{equation}
    where the implied constants depend on $\eta$ and $p$ but not on $r \in (0,1)$.
In addition, by Lemma \ref{conddens}, for sufficiently small $r$
\begin{align}\label{eq_y5}
\begin{aligned}
        &\mathbb{E}\big[\mathrm{1}_{|A+ r B| < r^{2-\eta}} \big| F(ir) = F(-ir) = 0\big] 
        \\
        &\qquad= \mathbb{E}\big[\mathrm{1}_{|\mathrm{Im}\big(F^{(0, 2)}(0) \cdot \overline{F^{(1, 0)}(0)}\big) + r\mathrm{Im}\big(i |F^{(0, 2)}(0)|^2 + \tfrac{1}{3} F^{(0, 3)}(0)\cdot\overline{F^{(1, 0)}(0)}\big)| < r^{2-\eta}}\big] \\
        &\qquad\lesssim \mathbb{P}\big(|\mathrm{Im}\big(Z_2 \overline{ Z_1}\big) +r \mathrm{Im}\big(i |Z_2|^2 + \tfrac{1}{3} Z_3 \overline{ Z_1}\big)| < r^{2-\eta}\big)
        \\
        &\qquad\lesssim r^{2-\eta} (1+|\log(r^{2-\eta})|),
    \end{aligned}
    \end{align}
    where $(Z_1,Z_2,Z_3)$ is a multiple of a standard complex random vector (the multiplying constant being absolute) and the last bound holds for sufficiently small $r$, as seen by a direct computation (see Lemma \ref{lemappxdirect} below).

    Combining \eqref{eq_y1}, \eqref{eq_y2}, \eqref{eq_y3}, \eqref{eq_y4} and \eqref{eq_y5} we obtain that, for sufficiently small $r$,
    \begin{align*}
        &e^+_r \lesssim (r^{2-\eta} (1+|\log(r^{2-\eta})|) + r^2)^{\frac{1}{p}}.
    \end{align*}
We now pick $\eta \in (0,1)$ and $p \in (1,\infty)$ such that $1+\alpha < \frac{2-\eta}{p}$ and $1+\alpha<2/p$, and obtain \eqref{eq_z1}. The argument for \eqref{eq_z2} is completely analogous, this time requiring the estimate
\[
\mathbb{P}\big(|\mathrm{Im}\big(Z_2 \overline{ Z_1}\big) -r \mathrm{Im}\big(i |Z_2|^2 - \tfrac{1}{3} Z_3 \overline{ Z_1}\big)| < r^{2-\eta}\big) \lesssim r^{2-\eta} (1+|\log(r^{2-\eta})|),
\]
 which we provide in Lemma \ref{lemappxdirect}.

\medskip
\noindent {\bf Step 3}. We consider the quantity
\begin{equation}\label{eq_x4}
d_r:=\mathrm{1}_{\jac F(ir) > 0} \mathrm{1}_{\jac F(-ir) > 0} - \mathrm{1}_{|A| < r B}
\end{equation}
and show that if $d_r \not= 0$, then
\begin{align}\label{eq_x2}
|A| \leq r (|B| + r |C_r^+| + r |C_r^-|)
\end{align}
and
\begin{align}\label{eq_x3}
\mbox{either } |A+ r B| \leq r^2|C_r^+|
\quad\mbox{or}\quad |-A + r B| \leq r^2 |C_r^-| 
\qquad\mbox{(possibly both).}
\end{align}
Assume that $d_r \not =0$. To see that \eqref{eq_x2} must hold, assume first that it does not. Then, $|A|>r|B|$, and, since, $d_r \not= 0$, we must have that $\mathrm{1}_{\jac F(ir) > 0} \mathrm{1}_{\jac F(-ir) > 0} = 1$. But, as seen in Step 1, this implies that \eqref{eq_x2} holds. Hence, \eqref{eq_x2}, must hold.

Second, note that, by Proposition \ref{condjac},
    \begin{align*}
        d_r &= \mathrm{1}_{r A + r^2 B + r^3 C_r^+ > 0} \mathrm{1}_{-r A + r^2 B + r^3 C_r^- > 0} - \mathrm{1}_{|A| < r B}
        \\
        &=\mathrm{1}_{A + r B + r^2 C_r^+ > 0} \mathrm{1}_{- A + r B + r^2 C_r^- > 0} -  \mathrm{1}_{A + r B > 0}\mathrm{1}_{-A + r B > 0}\neq 0.
    \end{align*}
This implies that
\begin{align*}
\mathrm{1}_{A + r B + r^2 C_r^+ > 0} \not= \mathrm{1}_{A + r B > 0}
\quad
\mbox{or}\quad \mathrm{1}_{-A + r B + r^2 C_r^- > 0} \not= \mathrm{1}_{-A + r B > 0}.
\end{align*}
In the first case, 
$r^2|C_r^+| = |(A+ r B)-(A+ r B+r^2C_r^+)|
\geq |A+ r B|$, while, in the second,
$r^2|C_r^-| = |(-A+ r B)-(-A+ r B+r^2C_r^-)|
\geq |-A+rB|$. Thus, \eqref{eq_x3} also holds.
    
\medskip\noindent {\bf Step 4}.
    We consider now
    \begin{align*}
        c_r&:=\mathbb{E}\big[(-A^2 + r^2 B^2)\mathrm{1}_{\jac F(ir) > 0} \mathrm{1}_{\jac F(-ir) > 0} \,\big |\, F(ir) = F(-ir) = 0 \big] \\
        &\qquad -\mathbb{E}\big[(-A^2 + r^2 B^2)\mathrm{1}_{|A| < r B}\,\big |\, F(ir) = F(-ir) = 0 \big] \\
        &= \mathbb{E}\big[(-A^2 + r^2 B^2)\cdot d_r \,\big|\, F(ir) = F(-ir) = 0\big],
    \end{align*}
    where $d_r$ is defined by \eqref{eq_x4}.
As a consequence of Step 3, 
\begin{align*}
|(-A^2 + r^2 B^2)\cdot d_r| \leq
r^2 \big((|B| + r |C_r^+| + r |C_r^-|)^2 + B^2 \big) \cdot \big( \mathrm{1}_{|A+ r B| < r^2 |C_r^+|} + \mathrm{1}_{|-A+ r B| < r^2 |C_r^-|} \big).
\end{align*}
Hence, we can use the estimates \eqref{eq_z1}, \eqref{eq_z2} to bound
    \begin{align*}
        |c_r| \leq r^2 (e_r^+ + e_r^-)
        \lesssim r^{3+\alpha}.
    \end{align*}
Combining this with \eqref{eq_z6} we conclude that
\begin{align}
\frac{\sigma^{c, +}(r)}{r^2} = \mathbb{E}\big[(-A^2 + r^2 B^2)\mathrm{1}_{|A| < r B} \,\big |\, F(ir) = F(-ir) = 0\big] +b_r + c_r,
\end{align}
with $|b_r + c_r| \lesssim r^{3+\alpha}$, as desired.
\end{proof}
The following lemma will help us analyze the approximate expression for $\sigma^{c,+}$ derived in Lemma \ref{approxp}.\footnote{Below and throughout $\mathrm{Im}(z)^2$ should be read as $\big(\mathrm{Im}(z)\big)^2$.}
\begin{lemma} \label{boundsp}
    For $z_1, z_2, z_3 \in \mathbb{C}$ and $r > 0$ let 
    \begin{equation*}
        \varphi_r(z_1, z_2, z_3) = \big(-\mathrm{Im}(z_2 \overline{z_1})^2 + r^2 \mathrm{Im}(i|z_2|^2 + \tfrac{1}{3} z_3\overline{z_1})^2\big)\mathrm{1}_{|\mathrm{Im}(z_2\overline{z_1})| < r\mathrm{Im}(i|z_2|^2 + \frac{1}{3} z_3\overline{z_1})}.
    \end{equation*}
    Then there exist constants $d,D>0$ such that if $(Z_1, Z_2, Z_3) \sim \mathcal{N}_{\mathbb{C}}(I)$ then
    \begin{equation*}
        d r^3 < \mathbb{E}[\varphi_r(Z_1, Z_2, Z_3)] < D r^3, \qquad 0<r<1. 
    \end{equation*}
\end{lemma}
\begin{proof}
    \noindent {\bf Step 1}. We first prove the lower bound. 
    For $z_1, z_2 \in \mathbb{C}$ we write $z_1 \overline{z_2} = |z_1| |z_2| e^{i \theta_{12}}$, where
    $\theta_{12} = \theta_{12}(z_1,z_2) \in (-\pi,\pi]$. If $z_1=0$ or $z_2=0$ we set  $\theta_{12} = 0$. 
    
    We consider the region
    \begin{align*}
        \Omega_r = \{(z_1,z_2,z_3) \in \mathbb{C}^3\,:\, |\sin(\theta_{12})| < r\} \cap (K_1 \times \mathbb{D} \times K_3),
    \end{align*}
    where $K_1, K_3$ are the rectangles 
    \begin{align*}
        &K_1 = \{z_1\in\mathbb{C}:\tfrac{1}{2} < \text{Re}(z_1) < \tfrac{1}{\sqrt{2}}, -\tfrac{1}{\sqrt{2}} < \mathrm{Im}(z_1) < -\tfrac{1}{2}\}, \\
        &K_3 =\{z_3\in\mathbb{C}:6 < \text{Re}(z_3) < 7, 6 < \mathrm{Im}(z_3) < 7\}.
    \end{align*}
    Then for $(z_1, z_2, z_3) \in \Omega_r$ we have
    \begin{equation*}
        |\mathrm{Im}(z_2\overline{z_1})| = |z_1| |z_2| |\sin(\theta_{12})| < r 
    \end{equation*}
    and
    \begin{equation*}
        \mathrm{Im}(i|z_2|^2 + \tfrac{1}{3} z_3\overline{z_1}) > \tfrac{1}{3} (6 \tfrac{1}{2} + 6 \tfrac{1}{2}) = 2.
    \end{equation*}
    Hence, for $(z_1, z_2, z_3) \in \Omega_r$,
    \begin{equation*}
        \varphi_r(z_1, z_2, z_3) > 3r^2.
    \end{equation*}
    Since $\Omega_r \subset K_1 \times \mathbb{D} \times K_3$ for all $r>0$, the probability density of $(Z_1, Z_2, Z_3)$ is bounded from below on $\Omega_r$ by a constant $c > 0$ independent of $r$. Thus, for $0 < r <1$,
    \begin{align*}
        &\mathbb{E}[\varphi_r(Z_1, Z_2, Z_3)] \geq \mathbb{E}[\varphi_r(Z_1, Z_2, Z_3)\mathrm{1}_{(Z_1, Z_2, Z_3) \in \Omega_r}] > 3r^2 \mathbb{E}[\mathrm{1}_{(Z_1, Z_2, Z_3) \in \Omega_r}] \\
        &\qquad> 3cr^2 \int_{\Omega_r} dA(z_1) dA(z_2) dA(z_3) = 3cr^2 \int_{K_1 \times \mathbb{D} \times K_3} \mathrm{1}_{|\sin(\theta_{12})| < r} dA(z_1) dA(z_2) dA(z_3) \\
        &\qquad\geq 3 c r^2 |K_3| \int_{K_1 \times \mathbb{D}} \mathrm{1}_{|\theta_{12}|< r} dA(z_1) dA(z_2) = 3 c r^2 |K_3| |K_1| \frac{2r}{2\pi} \pi  = 3c |K_1| |K_3| r^3,
    \end{align*}
    as, for fixed $z_1$, $\theta_{12}$ and the argument of $z_2$ differ only by a constant (modulo $2\pi$), which proves the lower bound.

\smallskip

    \noindent {\bf Step 2}. We consider now the  upper bound and note that
    \begin{align*}
        \varphi_r(z_1, z_2, z_3) &\leq 2r^2 \mathrm{Im}(i|z_2|^2 + \frac{1}{3} z_3\overline{z_1})^2 \mathrm{1}_{|\mathrm{Im}(z_2\overline{z_1})| < r\mathrm{Im}(i|z_2|^2 + \frac{1}{3} z_3\overline{z_1})} \\
        &\leq 4 r^2 (|z_1|^4 + |z_2|^4 +  |z_3|^4)\mathrm{1}_{|\mathrm{Im}(z_2\overline{z_1})| < r\mathrm{Im}(i|z_2|^2 + \frac{1}{3} z_3\overline{z_1})} \\
        &\leq 4r^2 (|z_1|^4 + |z_2|^4 +  |z_3|^4) \mathrm{1}_{|\mathrm{Im}(z_2\overline{z_1})| < 2r |z_2|^2} \\
        &\qquad+  4r^2(|z_1|^4 + |z_2|^4 +  |z_3|^4)\mathrm{1}_{|\mathrm{Im}(z_2\overline{z_1})| < 2r |z_1| |z_3|}.
    \end{align*}
    We also note that, for $c \geq 0$,
    \begin{equation*}
        \int_0^{2\pi} \mathrm{1}_{|\sin(\theta)| < c}  d\theta = \begin{cases} 4 \arcsin(c) \text{ if } c \leq 1 \\
        2\pi \text { if } c > 1\end{cases} \leq 8c  
    \end{equation*}
    and that for fixed $z_2$, $\theta_{12}$ and the argument of $z_1$ differ only by a constant (modulo $2\pi$). We estimate
    \begin{align*}
        &\mathbb{E}\big[4r^2 (|Z_1|^4 + |Z_2|^4 +  |Z_3|^4) \mathrm{1}_{|\mathrm{Im}(Z_2 \overline{Z_1})| < 2r |z_2|^2}\big] \\
        &\quad= \int_{\mathbb{C}^3} 4r^2 (|z_1|^4 + |z_2|^4 +  |z_3|^4)\mathrm{1}_{|\sin(\theta_{12})| |z_1| |z_2| < 2r |z_2|^2} \frac{1}{\pi^3} e^{-|z_1|^2 -|z_2|^2 - |z_3|^2} dA(z_1) dA(z_2) dA(z_3) \\
        &\quad= \int_{\mathbb{C}^2} \int_0^{\infty} \int_0^{2\pi} 4r^2(t_1^4 + |z_2|^4 + |z_3|^4) \mathrm{1}_{|\sin(\theta_{1})| < 2r \frac{|z_2|}{t_1}} \frac{1}{\pi^3}e^{-t_1^2 - |z_2|^2 - |z_3|^2} t_1 d\theta_1 dt_1 dA(z_2) dA(z_3) \\
        &\quad\leq r^3 \int_{\mathbb{C}^2} \int_0^{\infty} \frac{64}{\pi^3} (t_1^4 + |z_2|^4 + |z_3|^4) |z_2| e^{-t_1^2 - |z_2|^2 - |z_3|^2} dt_1 dA(z_2) dA(z_3).
    \end{align*}
    Similarly,
    \begin{align*}
        &\mathbb{E}\big[4r^2(|Z_1|^4 + |Z_2|^4 +  |Z_3|^4)\mathrm{1}_{|\mathrm{Im}(Z_2\overline{Z_1})| < 2r |Z_1| |Z_3|}\big] \\
        &\quad\leq r^3 \int_{\mathbb{C}^2} \int_0^{\infty} \frac{64}{\pi^3} (|z_1|^4 + t_2^4 + |z_3|^4) |z_3| e^{-|z_1|^2 - t_2^2 - |z_3|^2} dt_2 dA(z_1) dA(z_3).
    \end{align*}
    The upper bound now follows as
    \begin{align*}
        \mathbb{E}[\varphi_r(Z_1, Z_2, Z_3)] 
        &\leq \mathbb{E}\big[4r^2 (|Z_1|^4 + |Z_2|^4 +  |Z_3|^4) \mathrm{1}_{|\mathrm{Im}(Z_2 \overline{Z_1})| < 2r |Z_2|^2}\big]\\
        &\qquad+ \mathbb{E}\big[4r^2(|Z_1|^4 + |Z_2|^4 +  |Z_3|^4)\mathrm{1}_{|\mathrm{Im}(Z_2\overline{Z_1})| < 2r |Z_1| |Z_3|}\big] \lesssim r^3.
    \end{align*}
\end{proof}
Next, we combine all the previous asymptotics and describe the local correlations of critical points of GEF with positive indices.
\begin{prop} \label{repp}
    There exist constants $c, C > 0$ such that
    \begin{equation}\label{eq_ax}
        c \rho^7 \leq \mathbb{E}[\Ncp \cdot (\Ncp - 1)] \leq C \rho^7, \qquad 0<\rho<1.
    \end{equation}
\end{prop}
\begin{proof}
We prove \eqref{eq_ax} for sufficiently small $\rho$; the full claim then follows from the fact that $\mathbb{E}[\Ncp \cdot (\Ncp - 1)]$ is a increasing function of $\rho$, together with the bound $\mathbb{E}[\Ncp \cdot (\Ncp - 1)] \lesssim \rho^4$, given by Corollary \ref{coro_rho4}. Proposition \ref{approxKacRice} shows that
    \begin{align}\label{eq_ax3}
        \rho^2 \int_0^{\frac{\rho}{4}} \sigma^{c, +}(r) r^{-1} dr \lesssim \mathbb{E}[\Ncp \cdot (\Ncp - 1)] \lesssim \rho^2 \int_0^{\rho} \sigma^{c, +}(r) r^{-1} dr,
    \end{align} 
     where 
    \begin{align}
        \sigma^{c,+}(r) := \mathbb{E}\big[|\jac F(ir)\cdot\jac F(-ir)| \mathrm{1}_{\jac F(z) > 0} \mathrm{1}_{\jac F(w) > 0} \,\big|\, F(ir) = F(-ir)= 0\big],
    \end{align}
    which, by Lemma \ref{approxp}, satisfies the following expansion:
    \begin{align*}
        \sigma^{c, +}(r)  = r^2 \cdot \mathbb{E}\big[(-A^2 + r^2 B^2)\mathrm{1}_{|A| < r B} \,\big |\, F(ir) = F(-ir) = 0\big] + O(r^{5+1/2}), \qquad 0<r<1.
    \end{align*}
    We recall the definition of $A$ and $B$ and write
    \[
    (-A^2 + r^2 B^2) \mathrm{1}_{|A| < r B}=\varphi_r(F^{(1, 0)}(0), F^{(0, 2)}(0), F^{(0, 3)}(0)) 
    \]
    where 
      \begin{equation*}
        \varphi_r(z_1, z_2, z_3) = \big(-\mathrm{Im}(z_2 \overline{z_1})^2 + r^2 \mathrm{Im}(i|z_2|^2 + \tfrac{1}{3} z_3\overline{z_1})^2\big)\mathrm{1}_{|\mathrm{Im}(z_2\overline{z_1})| < r\mathrm{Im}(i|z_2|^2 + \tfrac{1}{3} z_3\overline{z_1})}.
    \end{equation*}
    We note that $\varphi_r : \mathbb{C}^3 \rightarrow [0, +\infty)$ is homogeneous of degree $4$ and use the approximate description of the conditional vector  
   $(F^{(1, 0)}(0), F^{(0, 2)}(0), F^{(0, 3)}(0)) \,|\, (F(ir) = F(-ir) = 0)$ given in Lemma \ref{conddens} to obtain
   \begin{align}\label{eq_ax2}
   \sigma^{c, +}(r) \asymp r^2 \cdot \mathbb{E}[\varphi_r(Z_1,  Z_2,  Z_3)] + O(r^{5+1/2}), \qquad 0<r<\delta,
   \end{align}
   where $\delta>0$ is an adequate constant and
   $(Z_1, Z_2, Z_3)$ is a standard complex normal vector.
    
By Lemma \ref{boundsp}, $\mathbb{E}[\varphi_r(Z_1,  Z_2,  Z_3)] \asymp r^3$, which, combined with \eqref{eq_ax2}, gives
    \begin{equation*}
        \sigma^{c, +}(r) \asymp r^5, \qquad 0<r<\delta.
    \end{equation*} 
    Inserting this into \eqref{eq_ax3}, shows that both sides of the estimate are $\rho^2 \int_0^{\frac{\rho}{4}} r^4 dr  \asymp \rho^2 \int_0^{\rho}r^4 dr  \asymp\rho^7$, for sufficiently small $\rho$, which finishes the proof of \eqref{eq_ax}.
\end{proof}

\section{Critical points with negative index}\label{sec_cns}
We look into the correlation among critical points of GEFs with negative index. While the analysis is similar to that of Section \ref{sec_cps}, some differences are substantial and we provide those details. We start by studying the approximate intensity given by Proposition \ref{approxKacRice}.
\begin{lemma} \label{approxm}
    Let $\sigma^{c, -}(r)$ be the approximate intensity \eqref{eq_crm} and $0 <\alpha < 1$. Then
    \begin{equation}\label{eq_mxx}
        \frac{\sigma^{c, -}(r)}{r^2}= \mathbb{E}\big[(-A^2 + r^2 B^2)\mathrm{1}_{|A| < -r B} \,\big |\, F(ir) = F(-ir) = 0\big] + O(r^{3+\alpha}), \qquad 0< r < 1,
    \end{equation}
\end{lemma}
where the implied constant depends on $\alpha$.
\begin{proof}
With the notation of Proposition, \ref{condjac}, we have
    \begin{align*}
        \frac{\sigma^{c, -}(r)}{r^2} = \mathbb{E}\big[(-A^2 + r^2 B^2) \mathrm{1}_{\jac F(ir) < 0} \mathrm{1}_{\jac F(-ir) < 0} \,\big |\, F(ir) = F(-ir) = 0 \big] + b_r',
    \end{align*}
    where
    \begin{align*}
        b_r' = \mathbb{E}\big[(r^2 A D_r + r^3 F_r) \mathrm{1}_{\jac F(ir) < 0} \mathrm{1}_{\jac F(-ir) < 0} \,\big |\, F(ir) = F(-ir) = 0 \big].
    \end{align*}

    \medskip \noindent {\bf Step 1}. We first show that 
    \begin{align}\label{eq_z6m}
    |b_r'| \lesssim r^{3+\alpha}.
    \end{align}
    Proposition \ref{condjac} implies that $|b_r'| \lesssim 1$ for $0<r<1$, so we can focus on the small $r$ range. We fix
    $0<\alpha<1$ and let constants depend on it. Inspecting the expansion of $\jac F(\pm ir)$ given by Proposition \ref{condjac}, we see that $\jac F(ir) < 0$ and $\jac F(-ir) < 0$ 
    occur simultaneously if and only if
    \begin{align*}
    -A> rB+r^2C_r^+ \mbox{ and }
    A> rB+r^2C_r^-.
    \end{align*}
    In this case
    $r B + r^2 C_r^- < A < -r B - r^2 C_r^+$, which gives
    \begin{equation*}
        |A| < r |B| + r^2 |C_r^+| + r^2 |C_r^-|.
    \end{equation*}
    Thus,
    \begin{align*}
        |b_r'| \leq r^3 \mathbb{E}\big[(|D_r|(|B| + r |C_r^+| + r |C_r^-|) + |F_r| ) \mathrm{1}_{|A| < r(|B| + r |C_r^+| + r |C_r^-|)} \big| F(ir) = F(-ir) = 0\big]
    \end{align*}   
    and, by Step 1 of the proof of Lemma \ref{approxp},  $|b_r'| \lesssim r^{3+\alpha}$. 

\medskip\noindent {\bf Step 2}. We note that by Step 2 of the proof of Lemma \ref{approxp}:
\begin{align}\label{eq_z1m}
e_r^+:=\mathbb{E}\big[\big((|B| + r |C_r^+| + r |C_r^-|)^2 +  B^2 \big) \mathrm{1}_{|A+ r B| < r^2 |C_r^+|} \,\big|\, F(ir) = F(-ir) = 0\big] \lesssim r^{1+\alpha},
\\\label{eq_z2m}
e_r^-:=\mathbb{E}\big[\big((|B| + r |C_r^+| + r |C_r^-|)^2 +  B^2 \big) \mathrm{1}_{|-A+ r B| < r^2 |C_r^-|} \,\big|\, F(ir) = F(-ir) = 0\big] \lesssim r^{1+\alpha}.
\end{align}

\medskip
\noindent {\bf Step 3}. We consider the quantity
\begin{equation}\label{eq_x4m}
d_r':=\mathrm{1}_{\jac F(ir) < 0} \mathrm{1}_{\jac F(-ir) < 0} - \mathrm{1}_{|A| < -r B}
\end{equation}
and show that if $d_r' \not= 0$, then
\begin{align}\label{eq_x2m}
|A| \leq r (|B| + r |C_r^+| + r |C_r^-|)
\end{align}
and
\begin{align}\label{eq_x3m}
\mbox{either } |A+ r B| \leq r^2|C_r^+|
\quad\mbox{or}\quad |-A + r B| \leq r^2 |C_r^-| 
\qquad\mbox{(possibly both).}
\end{align}
Assume that $d_r' \not =0$. To see that \eqref{eq_x2m} must hold, assume first that it does not. Then, $|A|>r|B| \geq -rB$, and, since, $d_r' \not= 0$, we must have that $\mathrm{1}_{\jac F(ir) < 0} \mathrm{1}_{\jac F(-ir) < 0} = 1$. But, as seen in Step 1, this implies that \eqref{eq_x2m} holds. Hence, \eqref{eq_x2m}, must hold.

Second, note that, by Proposition \ref{condjac},
    \begin{align*}
        d_r' &= \mathrm{1}_{r A + r^2 B + r^3 C_r^+ < 0} \mathrm{1}_{-r A + r^2 B + r^3 C_r^- < 0} - \mathrm{1}_{|A| < -r B}
        \\
        &=\mathrm{1}_{A + r B + r^2 C_r^+ < 0} \mathrm{1}_{- A + r B + r^2 C_r^- < 0} -  \mathrm{1}_{A + r B < 0}\mathrm{1}_{-A + r B < 0}\neq 0.
    \end{align*}
This implies that
\begin{align*}
\mathrm{1}_{A + r B + r^2 C_r^+ < 0} \not= \mathrm{1}_{A + r B < 0}
\quad
\mbox{or}\quad \mathrm{1}_{-A + r B + r^2 C_r^- < 0} \not= \mathrm{1}_{-A + r B < 0}.
\end{align*}
In the first case, 
$r^2|C_r^+| = |(A+ r B)-(A+ r B+r^2C_r^+)|
\geq |A+ r B|$, while, in the second,
$r^2|C_r^-| = |(-A+ r B)-(-A+ r B+r^2C_r^-)|
\geq |-A+rB|$. Thus, \eqref{eq_x3m} also holds.
    
\medskip\noindent {\bf Step 4}.
    We consider now
    \begin{align*}
        c_r' &:=\mathbb{E}\big[(-A^2 + r^2 B^2)\mathrm{1}_{\jac F(ir) < 0} \mathrm{1}_{\jac F(-ir) < 0} \,\big |\, F(ir) = F(-ir) = 0 \big] \\
        &\qquad -\mathbb{E}\big[(-A^2 + r^2 B^2)\mathrm{1}_{|A| < -r B}\,\big |\, F(ir) = F(-ir) = 0 \big] \\
        &= \mathbb{E}\big[(-A^2 + r^2 B^2)\cdot d_r' \,\big|\, F(ir) = F(-ir) = 0\big],
    \end{align*}
    where $d_r'$ is defined by \eqref{eq_x4m}.
As a consequence of Step 3,
\begin{align*}
\left|(-A^2 + r^2 B^2)\cdot d_r'\right| \leq
r^2 \big((|B| + r |C_r^+| + r |C_r^-|)^2 + B^2 \big) \cdot \big( \mathrm{1}_{|A+ r B| < r^2 |C_r^+|} + \mathrm{1}_{|-A+ r B| < r^2 |C_r^-|} \big).
\end{align*}
Hence, we can use the estimates \eqref{eq_z1m}, \eqref{eq_z2m} to bound
    \begin{align*}
        |c_r'| \leq r^2 (e_r^+ + e_r^-)
        \lesssim r^{3+\alpha}.
    \end{align*}
Combining this with \eqref{eq_z6m} we conclude that
\begin{align}
 \frac{\sigma^{c, -}(r)}{r^2} = \mathbb{E}\big[(-A^2 + r^2 B^2)\mathrm{1}_{|A| < -r B} \,\big |\, F(ir) = F(-ir) = 0\big] +b_r' + c_r',
\end{align}
with $|b_r' + c_r'| \lesssim r^{3+\alpha}$, as desired.
\end{proof}
The following analogue of Lemma \ref{boundsp} helps to analyze the expression in \eqref{eq_mxx}.

\begin{lemma} \label{boundsm}
    For $z_1, z_2, z_3 \in \mathbb{C}$ and $r > 0$ let 
    \begin{equation}\label{eq_phiprime}
        \varphi_r'(z_1, z_2, z_3) = \big(-\mathrm{Im}(z_2 \overline{z_1})^2 + r^2 \mathrm{Im}(i|z_2|^2 + \tfrac{1}{3} z_3\overline{z_1})^2\big)\mathrm{1}_{|\mathrm{Im}(z_2\overline{z_1})| < -r\mathrm{Im}(i|z_2|^2 + \frac{1}{3} z_3\overline{z_1})}.
    \end{equation}
    Then there exist constants $d,D>0$ such that if $(Z_1, Z_2, Z_3) \sim \mathcal{N}_{\mathbb{C}}(I)$ then
    \begin{equation*}
        d r^3 < \mathbb{E}[\varphi_r'(Z_1, Z_2, Z_3)] < D r^3, \qquad 0<r<1. 
    \end{equation*} 
\end{lemma}
\begin{proof}
    \noindent {\bf Step 1}. We first prove the lower bound. As in the proof of Lemma \ref{boundsp}, for $z_1, z_2 \in \mathbb{C}$ we write $z_1 \overline{z_2} = |z_1| |z_2| e^{i \theta_{12}}$, with
    $\theta_{12} = \theta_{12}(z_1,z_2) \in (-\pi,\pi]$. We consider the region
    \begin{align*}
        \Omega_r' = \{(z_1,z_2,z_3) \in \mathbb{C}^3\,:\, |\sin(\theta_{12})| < r\} \cap (K_1 \times \mathbb{D} \times K_3'),
    \end{align*}
    where $K_1, K_3'$ are the rectangles 
    \begin{align*}
        &K_1 = \{z_1\in\mathbb{C}:\tfrac{1}{2} < \text{Re}(z_1) < \tfrac{1}{\sqrt{2}}, -\tfrac{1}{\sqrt{2}} < \mathrm{Im}(z_1) < -\tfrac{1}{2}\}, \\
        &K_3' =\{z_3\in\mathbb{C}: -13 < \text{Re}(z_3) < -12, -13 < \mathrm{Im}(z_3) < -12\}.
    \end{align*}
    Then for $(z_1, z_2, z_3) \in \Omega_r$ we have
    \begin{equation*}
        |\mathrm{Im}(z_2\overline{z_1})| = |z_1| |z_2| |\sin(\theta_{12})| < r 
    \end{equation*}
    and
    \begin{equation*}
        -\mathrm{Im}(i|z_2|^2 + \tfrac{1}{3} z_3\overline{z_1}) > - 1 + \tfrac{1}{3} (12 \tfrac{1}{2} + 12 \tfrac{1}{2}) = 3.
    \end{equation*}
    Hence, for $(z_1, z_2, z_3) \in \Omega_r$,
    \begin{equation*}
        \varphi'_r(z_1, z_2, z_3) > 8r^2.
    \end{equation*}
    The probability density of $(Z_1, Z_2, Z_3)$ is bounded from below on $\Omega_r'$ by a constant $c > 0$ independent of $r \in (0,1)$ and, thus,
    \begin{align*}
        &\mathbb{E}[\varphi'_r(Z_1, Z_2, Z_3)] \geq \mathbb{E}[\varphi'_r(Z_1, Z_2, Z_3)\mathrm{1}_{(Z_1, Z_2, Z_3) \in \Omega_r}] > 8r^2 \mathbb{E}[\mathrm{1}_{(Z_1, Z_2, Z_3) \in \Omega_r'}] \\
        &\qquad> 8cr^2 \int_{\Omega_r'} dA(z_1) dA(z_2) dA(z_3) = 8cr^2 \int_{K_1 \times \mathbb{D} \times K_3'} \mathrm{1}_{|\sin(\theta_{12})| < r} dA(z_1) dA(z_2) dA(z_3) \\
        &\qquad\geq 8 c r^2 |K_3'| \int_{K_1 \times \mathbb{D}} \mathrm{1}_{|\theta_{12}|< r} dA(z_1) dA(z_2) = 8 c r^2 |K_3'| |K_1| \frac{2r}{2\pi} \pi  = 8c |K_1| |K_3'| r^3,
    \end{align*}
    as, for fixed $z_1$, $\theta_{12}$ and the argument of $z_2$ differ only by a constant (modulo $2\pi$), which proves the lower bound.

\smallskip

    \noindent {\bf Step 2}. We consider now the  upper bound and note that
    \begin{align*}
        \varphi_r'(z_1, z_2, z_3) &\leq 2r^2 \mathrm{Im}(i|z_2|^2 + \tfrac{1}{3} z_3\overline{z_1})^2 \mathrm{1}_{|\mathrm{Im}(z_2\overline{z_1})| < -r\mathrm{Im}(i|z_2|^2 + \frac{1}{3} z_3\overline{z_1})} \\
        &\leq 4 r^2 (|z_1|^4 + |z_2|^4 +  |z_3|^4)\mathrm{1}_{|\mathrm{Im}(z_2\overline{z_1})| < -r\mathrm{Im}(i|z_2|^2 + \frac{1}{3} z_3\overline{z_1})} \\
        &\leq 4r^2(|z_1|^4 + |z_2|^4 +  |z_3|^4)\mathrm{1}_{|\mathrm{Im}(z_2\overline{z_1})| < 2r |z_1| |z_3|},
    \end{align*}
    and the upper bound follows by the proof of Step 2 of Lemma \ref{boundsp}.
\end{proof}
Collecting all the previous results, we can now describe the local correlations between critical points of GEF with negative indices.
\begin{prop} \label{repm}
    There exist constants $c, C > 0$ such that
    \begin{equation}\label{eq_axm}
        c \rho^7 \leq \mathbb{E}[\Ncm\cdot(\Ncm - 1)] \leq C \rho^7, \qquad
        0<\rho<1.
    \end{equation}
\end{prop}
\begin{proof}
As before, the monotonicity of $\mathbb{E}[\Ncm\cdot(\Ncm - 1)]$ on $\rho$ and the bound $\mathbb{E}[\Ncm \cdot (\Ncm - 1)] \lesssim \rho^4$, given by 
Corollary \ref{coro_rho4} allows us to reduce the analysis to small $\rho$.
By Proposition \ref{approxKacRice},
    \begin{align}\label{eq_ax3m}
        \rho^2 \int_0^{\frac{\rho}{4}} \sigma^{c, -}(r) r^{-1} dr \lesssim \mathbb{E}[\Ncm \cdot (\Ncm - 1)] \lesssim \rho^2 \int_0^{\rho} \sigma^{c, -}(r) r^{-1} dr,
    \end{align} 
     where 
    \begin{align}
        \sigma^{c,-}(r) := \mathbb{E}\big[|\jac F(ir)\cdot\jac F(-ir)| \mathrm{1}_{\jac F(z) < 0} \mathrm{1}_{\jac F(w) < 0} \,\big|\, F(ir) = F(-ir)= 0\big],
    \end{align}
    which, by Lemma \ref{approxm}, satisfies the following expansion:
    \begin{equation*}
        \sigma^{c, -}(r)  = r^2 \mathbb{E}\big[(-A^2 + r^2 B^2)\mathrm{1}_{|A| < -r B} \,\big |\, F(ir) = F(-ir) = 0\big] + O(r^{5+1/2}), \qquad 0<r<1.
    \end{equation*}
    We recall the definition of $A$ and $B$ and write
    \begin{align}
        (-A^2 + r^2 B^2) \mathrm{1}_{|A| < - rB} = \varphi_r'(F^{(1, 0)}(0), F^{(0, 2)}(0), F^{(0 ,3)}(0))
    \end{align}
     where $\varphi'_r: \mathbb{C}^3 \rightarrow [0, +\infty)$ is given by \eqref{eq_phiprime}. We note that $\varphi_r'$ is homogeneous of degree $4$ and use the approximate description of the conditional vector  
   $(F^{(1, 0)}(0), F^{(0, 2)}(0), F^{(0, 3)}(0)) \,|\, (F(ir) = F(-ir) = 0)$ given in Lemma \ref{conddens} to obtain
   \begin{align}\label{eq_ax2m}
   \sigma^{c, -}(r) \asymp r^2 \mathbb{E}[\varphi'_r(Z_1,  Z_2,  Z_3)] + O(r^{5+1/2}), \qquad 0<r<\delta,
   \end{align}
   where $\delta>0$ is an adequate constant and $(Z_1, Z_2, Z_3)$ is a standard complex normal vector.
    By Lemma \ref{boundsm} $\mathbb{E}[\varphi'_r(Z_1,  Z_2,  Z_3)] \asymp r^3$ for $0 < r< 1$, which, combined with \eqref{eq_ax2m}, gives
    \begin{equation*}
        \sigma^{c, -}(r) \asymp r^5, \qquad 0<r<\delta.
    \end{equation*} 
    Inserting this into \eqref{eq_ax3m}, shows that both sides of the estimate are $\rho^2 \int_0^{\frac{\rho}{4}} r^4 dr  \asymp \rho^2 \int_0^{\rho}r^4 dr  \asymp\rho^7$, provided that $\rho$ is sufficiently small, and completes the proof of \eqref{eq_axm}.
\end{proof}

\section{Zeros and positively signed critical points}\label{sec_zp1}

We now look into correlations between
zeros and critical points of $G$. We shall start by considering critical points with positive index and invoke the Kac-Rice formula:
\begin{align}\label{eq_kra}
\mathbb{E}[\Nz \Ncp] = \int_{B_\rho^2}  \frac{\mathbb{E}\big[|\jac G(z) \, \jac F(w)| \, \mathrm{1}_{\jac F(w) > 0} \,\big|\, G(z) = F(w) = 0\big]}{\pi^2 e^{|z|^2 + |w|^2}\big(1- e^{-|z-w|^2} |z-w|^2\big)} dA(z)dA(w),
\end{align}
see Lemma \ref{kacrice} for details. We note that (in contrast with the corresponding expressions for correlations between critical points) the denominator in \eqref{eq_kra} is bounded above and below
by positive constants that are independent of $z,w$ as soon as $\rho$ is sufficiently small. For the numerator
in \eqref{eq_kra}, we shall make use of the following elementary identities, which rely on the analyticity of $G$:
\begin{align}\label{eq_m1}
&\jac G(z) = 
|\partial G(z)|^2 - |\bar{\partial} G(z)|^2=
|\partial G(z)|^2,
\\\label{eq_m2}
&\jac F(0) = 
|\partial F(0)|^2 - |\bar{\partial} F(0)|^2=
|\partial^2 G(0)|^2- |G(0)|^2.
\end{align}
Following \cite{ladgham2023local}, we introduce terminology to describe asymptotic expansions more succinctly. For $k\geq0$, we denote by $\OrdoP(\rho^k)$ a random variable $X$ such that the absolute value of the conditioned variable 
$(X \,\big|\, F(z)=G(w)=0)$ 
is $\leq \rho^k Y$, where 
$Y \geq 0$ has stretched exponential tails with constants independent of $z$ and $w$ in some given neighborhood $B_\rho(0)$ of the origin --- that is,
\begin{align*}
\sup_{|z|,|w|<\rho} \mathbb{P}\big[ \rho^{-k}|X| > t\,|\,F(z)=G(w)=0\big] \leq K e^{-k t^\gamma}, \qquad \mbox{for all }t>0,
\end{align*}
and some constants $k,K,\gamma>0$. (We will
only use this notation in contexts where the conditioning in question is well-defined). Repeating the argument of Proposition \ref{condjac}, the Taylor expansions of $G$ and $F$ have $\OrdoP(\rho^k)$ error terms for adequate $k$, a fact that will be used without further mention. In particular, for each $k \geq 0$, $\partial^k G(z)$ and $\partial^k F(z)$ are $\OrdoP(1)$ if $|z|<\rho$. We will also use repeatedly that
\[
\OrdoP(\rho^k)\cdot \OrdoP(\rho^l)= \OrdoP(\rho^{k+l}). 
\]
Second, we say that two random variables $X, Y$ are \emph{equal under a certain event $B$} if $X \cdot \mathrm{1}_B = T \cdot \mathrm{1}_B$. Note that if $X=Y$ under $B_1$ and also under $B_2$, then $X=Y$ under $B_1 \cup B_2$.

After this preparation, we can prove the following.

\begin{prop}\label{prop_zcp}
There exist constants $c, C > 0$ such that
\begin{align}\label{eq_zcpa}
c \rho^6 \leq 
\mathbb{E}[\Nz \cdot \Ncp] \leq C \rho^6, \quad 0 < \rho < 1.
\end{align}
\end{prop}
\begin{proof}
The denominator in \eqref{eq_kra} is bounded below, while, for $|z|,|w| \leq 1$,
$\jac G(z)$ and $\jac F(w)$ are $\OrdoP(1)$ conditionally on $F(z)=F(w)=0$. Thus the numerator in \eqref{eq_kra} is bounded, and we conclude that $\mathbb{E}[\Nz \cdot \Ncp]$ is bounded and depends monotonically on $\rho$. Hence, it is enough to prove \eqref{eq_zcpa} for small $\rho$.

\noindent {\bf Step 1}. 
As discussed, for sufficiently small $\rho$, the Kac-Rice formula \eqref{eq_kra} reduces asymptotically to 
\begin{align}\label{eq_kram}
\mathbb{E}[\Nz \Ncp] \asymp 
\int_{B_\rho^2} {\mathbb{E}\big[|\partial G(z)|^2 \cdot \jac F(w)| \cdot \mathrm{1}_{\jac F(w) > 0}\, \big| \,G(z) = F(w) = 0\big]}dA(z)dA(w).
\end{align}
We start by writing some Taylor expansions, always for $z, w \in B_\rho$ and
under the condition $G(z)=F(w)=0$: 
\[
\begin{split}
&\partial G(z)= \partial G(w)+ (z-w) \partial^2 G(w) + (z-w)^2E_1^{z,w}
\\
&= (\bar w G(w) - F(w)) + (z-w) \partial^2 G(w) + (z-w)^2E_1^{z,w} 
\\
&= \bar w G(w) + (z-w) \partial^2 G(w) + (z-w)^2E_1^{z,w} 
\\
&= \bar w \big[G(z) + (z-w)E_2^{z,w} \big] + (z-w) \partial^2 G(w)+ (z-w)^2E_1^{z,w} \\
&= (z-w) \partial^2 G(w) + (z-w)^2 E_1^{z,w}+ \bar w (z-w) E_2^{z,w}, 
\end{split}
\]
where the error terms are $\OrdoP(1)$. Consequently, 
\[
|\partial G(z)|^2= |z-w|^2 |\partial^2 G(w)|^2 + \Ordo_{\mathbb{P}}(\rho^3). 
\]
We also have (using $G(z)=F(w)=0$) that
\begin{align} \label{eq:jacF_exp}
&\jac F(w) = |\partial F(w)|^2- |\bar \partial F(w)|^2= |\partial^2 G(w) -\bar w \partial G(w)|^2- |G(w)|^2 \\
&\quad= |\partial^2 G(w)- \bar w^2 G(w)|^2-|G(w)|^2 \\
&\quad=
|\partial^2 G(w)|^2- 2 \mathrm{Re}(\partial^2G(w) w^2 \overline{G(w)})+ (|w|^4-1)|G(w)|^2 \\
&\quad=
|\partial^2 G(w)|^2- 2 \mathrm{Re}(\partial^2G(w) w^2 \overline{G(w)})+ (|w|^4-1)|(z-w) E_3^{z,w}|^2 \\
&\quad=|\partial^2 G(w)|^2 + \OrdoP(\rho),
\end{align}
where the term $E^{z,w}_3$ is defined by $0=G(z)=G(w)+(z-w)E^{z,w}_3$. Therefore 
\begin{align}\label{eq_very}
\begin{aligned}
|\partial G(z)|^2 \cdot |\jac F(w)| 
&= \big( |z-w|^2 \cdot |\partial^2 G(w)|^2+ \OrdoP(\rho^3) \big)\cdot \big( |\partial^2 G(w)|^2 + \OrdoP(\rho)\big) \\
&= 
|z-w|^2 \cdot |\partial^2 G(w)|^4 + \OrdoP(\rho^3).
\end{aligned}
\end{align}
\smallskip

\noindent {\bf Step 2}.
We consider the factor 
$\mathrm{1}_{\jac F(w) > 0}$ and use \eqref{eq:jacF_exp} to write 
\begin{align}\label{eq_very2}
\jac F(w)= |\partial^2 G(w)|^2 + E_4^{z,w}, 
\end{align}
with $E_4^{z,w}= \OrdoP(\rho)$. 
Since
\[
\big| \mathrm{1}_{\jac F(w)>0} - 1 \big| 
= \mathrm{1}_{\jac F(w) \leq 0} \leq 
\mathrm{1}_{|\partial^2 G(w)|^2 \leq |E^{z,w}_4|},
\]
we can estimate
\begin{align*}
&\mathbb{E}\big[|\partial G(z)|^2 \jac F(w) \mathrm{1}_{\jac F(w) > 0} \,\big| \,  G(z) = F(w) = 0 \big] \\
&
=\mathbb{E}\big[ \big(|z-w|^2|\partial^2 G(w)|^4 + \OrdoP(\rho^3) \big)(\mathrm{1}_{\jac F(w)>0}-1) \,\big| \, G(z)=F(w)=0 \big] \\
&\qquad+ \mathbb{E}\big[ \big(|z-w|^2|\partial^2 G(w)|^4 + \OrdoP(\rho^3) \big) \,\big| \, G(z)=F(w)=0 \big]
\\
&=|z-w|^2 \cdot \mathbb{E}\big[ |\partial^2 G(w)|^4(\mathrm{1}_{\jac F(w)>0}-1) \,\big| \, G(z)=F(w)=0 \big] \\
&\qquad+ |z-w|^2 \cdot \mathbb{E}\big[ |\partial^2 G(w)|^4  \,\big| \, G(z)=F(w)=0 \big]+\Ordo(\rho^3),
\end{align*}
where we used that $
\E \big[ \OrdoP(\rho^k) \,\big|\,G(z)=F(w)=0 \big]=\Ordo(\rho^k)$.
Hence,
\begin{align*}
\Big| &\mathbb{E}\big[|\partial G(z)|^2 \jac F(w) \mathrm{1}_{\jac F(w) > 0} \,\big| \,  G(z) = F(w) = 0\big] \\
 &\qquad-|z-w|^2 \cdot \mathbb{E}\big[|\partial^2 G(w)|^4  \, \big| \, G(z)=F(w)=0 \big]  \Big| \\
 &\leq 
 |z-w|^2 \cdot \mathbb{E}\big[ |\partial^2 G(w)|^4\,|\mathrm{1}_{\jac F(w)>0}-1| \,\big| \, G(z)=F(w)=0 \big] + \Ordo(\rho^3)
 \\
  &\leq 
 |z-w|^2 \cdot \mathbb{E}\big[ |\partial^2 G(w)|^4 \mathrm{1}_{|\partial^2 G(w)|^2 \leq |E^{z,w}_4|} \,\big| \, G(z)=F(w)=0 \big] + \Ordo(\rho^3)
 \\
&\leq 
|z-w|^2 \cdot \mathbb{E}\big[|E_4^{z,w}|^2  \,\big| \, G(z)=F(w)=0 \big] +  \Ordo(\rho^3) \\
&= 
\Ordo(\rho^3).
\end{align*}
That is,
\begin{align}\label{eq_very3}
&\mathbb{E}\big[|\partial G(z)|^2 \cdot \jac F(w) \cdot \mathrm{1}_{\jac F(w) > 0} \,\big| \,  G(z) = F(w) = 0\big] 
\\
&\quad= |z-w|^2 \cdot \mathbb{E}\big[|\partial^2 G(w)|^4  \, \big| \, G(z)=F(w)=0 \big] + \Ordo(\rho^3).
\end{align}

\noindent {\bf Step 3}. We look into 
\[
\mathbb{E}\big[|\partial^2 G(w)|^4  \, \big| \, G(z)=F(w)=0 \big].
\]
Recall  that $\big(G(0), \partial G(0), \tfrac12 \partial^2 G(0) \big)$ is a standard complex Gaussian vector, so the covariance matrix of 
$\big( G(z), F(w), \partial^2 G(w) \big)$ is of the form 
\[
\begin{pmatrix} 1 & 0&0 \\
0 &1 & 0 \\
0 & 0 & 2 \end{pmatrix} + \Ordo(\rho).
\]
Using the Gaussian regression formula for sufficiently small $\rho$ (see, e. g., \cite[Eq. 1.5]{level}), we have that 
\begin{align*}
\mathrm{Var} \big[ \partial^2 G(w)\,\big|\, G(z)=F(w)=0\big] = 2+ \Ordo(\rho).
\end{align*}
Since $\partial^2 G(w)$ is a zero mean complex random variable, we then have,
\begin{align}\label{eq_very4}
\mathbb{E}\big[|\partial^2 G(w)|^4  \, \big| \, G(z)=F(w)=0 \big] = 8+ \Ordo(\rho).
\end{align}
Combining \eqref{eq_very3} and \eqref{eq_very4}, we have obtained
\begin{align*}
&\mathbb{E}[|\partial G(z)|^2 | \cdot \jac F(w)| \cdot \mathrm{1}_{\jac F(w) > 0} \,\big| \,  G(z)= F(w) = 0]
\\&\quad= |z-w|^2( 8+ \Ordo(\rho))+ 
\Ordo( \rho^3) = 8 |z-w|^2 + \Ordo(\rho^3).
\end{align*}
Inserting this into the approximate Kac-Rice formula \eqref{eq_kram} gives
\[
\mathbb{E}[\Nz \Ncp] \asymp
4\int_{B_\rho^2} |z-w|^2 dz dw + \Ordo(\rho^7) \asymp \rho^6. 
\]
\end{proof}

\section{Zeros and negatively signed critical points}\label{sec_zp2}
Finally, we look into correlations between between zeros and critical points with negative index. This time, the Kac-Rice formula reads: 
\begin{align}\label{eq_krb}
 \mathbb{E}[\Nz \cdot \Ncm]= \int_{B_\rho \times B_\rho} \frac{1}{\pi^2}
  \frac{\mathbb{E}[|\partial G(z)|^2 \cdot |\jac F(w)| \cdot \mathrm{1}_{\jac F(w) < 0} \,\big| \,  G(z) = F(w) = 0]}{e^{|z|^2 + |w|^2}\big(1- |z-w|^2e^{-|z-w|^2}\big)}dA(z)dA(w),
\end{align}
see Lemma \ref{kacrice} for details. The integrand in \eqref{eq_krb} depends only on $z-w$ --- see Lemma \ref{lemmashift} --- and we will exploit this fact to reduce analysis to the case $w=0$. As before, we can asymptotically neglect the denominator. The analysis of the numerator is a bit more involved than in the case of critical points with positive index, and will require a series of lemmas.

\begin{lemma} \label{lem:minmax1}
For all sufficiently small $\rho>0$ and all $|z|<\rho$, we have, conditionally on $G(z)=F(0)=0$, that the following equalities hold under either the event $\{\jac F(0)<0\}$ or the event $\{|\partial^2 G(0)|^2 \leq \tfrac{1}{36} |z|^6 |\partial^3 G(0)|^2\}$:
\begin{equation} \label{eq:d2G_bound}
\partial^2 G(0) = \OrdoP( \rho^3),
\end{equation}
\begin{equation} \label{eq:jacF_bound}
\jac F(0) = |\partial^2 G(0)|^2 - \tfrac{1}{36} |z|^6 |\partial^3 G(0)|^2 + \OrdoP(\rho^7),
\end{equation}
and
\begin{equation} \label{eq:dG_bound}
|\partial G(z)|^2= \tfrac14 |z|^4 |\partial^3 G(0)|^2 + \OrdoP(\rho^5). 
\end{equation}
\end{lemma}
\begin{rem}
Recall that, according to our terminology, the meaning of \eqref{eq:d2G_bound} is that there exist constants $\alpha, \beta, C>0$ such that 
\[
\mathbb{P}\big[ |\partial^2 G(0)|\cdot 1_B > t\,\big|\, G(z)=F(0)=0\big] \leq C e^{-\alpha t^{\beta}}, \qquad t >0, |z|<\rho,
\]
where $B=\{\jac F(0)<0\}$ or $B=\{|\partial^2 G(0)|^2 \leq \tfrac{1}{36} |z|^6 |\partial^3 G(0)|^2\}$. Similar remarks apply to the other parts of Lemma \ref{lem:minmax1}. It is also possible to formulate Lemma \ref{lem:minmax1} in terms of conditioning on $B \cup \{G(z)=F(0)=0\}$.
\end{rem}
\begin{proof}[Proof of Lemma \ref{lem:minmax1}]
\noindent {\bf Step 1}. We assume throughout that $G(z)=F(0)=0$. Using that $\partial G(0)=-F(0)=0$, we have
\[
0 = G(0)+ \tfrac12 z^2 \partial^2 G(0) + \tfrac16 z^3 \partial^3 G(0) + \OrdoP(\rho^4), 
\]
so
\begin{equation} \label{eq:G_expansion}
G(0) =- \tfrac12 z^2 \partial^2 G(0) - \tfrac16 z^3 \partial^3 G(0) + \OrdoP(\rho^4),
\end{equation}
and 
\begin{align} \label{eq:G_modsquare}
|G(0)|^2 &= \tfrac14 |z|^4 |\partial^2 G(0)|^2 + \frac{1}{36}
|z|^6 \big| \partial^3 G(0)|^2 + \frac{1}{6}\mathrm{Re}\big(z^2 \partial^2G(0) \bar z^3 \overline{\partial^3G(0)}\big)+\OrdoP(\rho^6).
\end{align}
\smallskip

\noindent {\bf Step 2}. We assume now that $\jac F(0)<0$, that is, we assume that $|\partial^2 G(0)|^2 < |G(0)|^2$. In terms of \eqref{eq:G_modsquare} this means that
\[ \label{eq:sec_der_bound1}
|\partial^2 G(0)|^2 < \tfrac14 |z|^4 |\partial^2 G(0)|^2 + \frac{1}{36}
|z|^6 \big| \partial^3 G(0)|^2 + \frac{1}{6}|z|^5 |\partial^2G(0)| |\partial^3G(0)| +\OrdoP(\rho^6). 
\]
We let the LHS absorb the term $\tfrac14 |z|^4 |\partial^2 G(0)|^2$ --- which can be done for small enough $\rho$ --- with the result  
\[
|\partial^2 G(0)|^2 \leq |\partial^2 G(0)| \OrdoP( \rho^5) + \OrdoP( \rho^6),
\]
or, in other words,
\[ \label{eq:sec_der_bound2}
|\partial^2 G(0)|^2 \leq ( |\partial^2 G(0)| \rho^5 + \rho^6) E^z,
\]
where $E^z$ is a suitable function with stretched exponential tails when conditioned to $G(z)=F(0)=0$ (and the corresponding constants are independent of $z \in B_\rho(0)$). For almost every realization of $G$ we must have
$|\partial^2 G(0)|^2 \leq 2 |\partial^2 G(0)| \rho^5 E^z$ or $|\partial^2 G(0)|^2 \leq 2  \rho^6 E^z$, which implies that 
\begin{align*}
|\partial G(0)| \leq 2\rho^5 E^z+ \sqrt{2} \rho^3 (E^z)^{1/2}
= \rho^3 \big( 2\rho^2 E^z + \sqrt{2} (E^z)^{1/2}\big).
\end{align*}
This proves 
\eqref{eq:d2G_bound} under $\{\jac F(0)<0\}$. On the other hand, under $\{|\partial^2 G(0)|^2 \leq \tfrac{1}{36} |z|^6 |\partial^3 G(0)|^2\}$ we have
$|\partial^2 G(0)|^2 = \OrdoP(\rho^6)$ and \eqref{eq:d2G_bound} also holds.

\smallskip

\noindent {\bf Step 3}. Under either of the two events in question,  we reinspect \eqref{eq:G_expansion} using that $\partial^2 G(0) = \OrdoP(\rho^3)$ and find out that \eqref{eq:G_modsquare} can be improved to
\begin{align} \label{eq:G_modsquare2}
|G(0)|^2 &= \tfrac14 |z|^4 |\partial^2 G(0)|^2 + \frac{1}{36}
|z|^6 \big| \partial^3 G(0)|^2 + \frac{1}{6}\mathrm{Re}\big(z^2 \partial^2G(0) \bar z^3 \overline{\partial^3G(0)}\big)+\OrdoP(\rho^7)
\\
&=\frac{1}{36}
|z|^6 \big| \partial^3 G(0)|^2 + \OrdoP(\rho^7).
\end{align}
A direct computation now gives \eqref{eq:jacF_bound}:
\begin{equation} \label{eq:JacF_exp2}
\jac F(0) = |\partial^2 G(0)|^2- |G(0)|^2 
= |\partial^2 G(0)|^2- \frac{1}{36}|z|^6 |\partial^3 G(0)|^2 + \OrdoP(\rho^7).
\end{equation}

\smallskip

\noindent {\bf Step 4}. Under either of the two events in question, we Taylor expand $\partial G(z)$ around $0$. Taking into account \eqref{eq:d2G_bound}, we get
\begin{equation} \label{eq:partialG_exp}
\partial G(z) =  z \partial^2 G(0) + 
\frac12 z^2 \partial^3 G(0)+ \OrdoP(\rho^3)= \frac12 z^2 \partial^3 G(0) + \OrdoP(\rho^3). 
\end{equation}
This implies
\[
|\partial G(z)|^2= \frac14 |z|^4 |\partial^3 G(0)|^2 + \OrdoP(\rho^5),
\]
which is \eqref{eq:dG_bound}.
\end{proof}

\begin{lemma} \label{lem_minmax2}
For sufficiently small $\rho>0$, we have 
\begin{align}
&\mathbb{E} \big[ |\partial G(z)|^2 \cdot |\jac F(0)| \cdot \mathrm{1}_{\jac F(0) <0 } \,\big|\, G(z)=F(0)=0 \big]  \\
&= \E \Big[ \big| \tfrac12 z^2 \partial^3 G(0) \big|^2 \cdot \big| |\partial^2 G(0)|^2 - \tfrac{1}{36} |z|^6 |\partial^3 G(0)|^2 \big| \cdot
\mathrm{1}_{|\partial^2 G(0)|^2 < \tfrac{1}{36} |z|^6 |\partial^3 G(0)|^2} \,\big|\, G(z)=F(0)=0 \Big] \\
&\qquad+ \Ordo(\rho^{16.5}),
\end{align}
where implied constant is independent of $z \in B_\rho(0)$.
\end{lemma}
\begin{proof}
\noindent \textbf{Step 1}. We show that
\begin{align}\label{wes1}
&\mathbb{E} \big[ |\partial G(z)|^2 \cdot |\jac F(0)| \cdot \mathrm{1}_{\jac F(0) <0 } \,\big|\, G(z)=F(0)=0 \big]  \\
&
\quad= \mathbb{E} \big[
\tfrac14 |z|^4 |\partial^3 G(0)|^2 \cdot \big| |\partial^2 G(0)|^2 - \tfrac{|z|^6}{36} |\partial^3 G(0)|^2 \big| \cdot \mathrm{1}_{\jac F(0)<0}\,\big|\, G(z)=F(0)=0
\big]
\\
&\qquad\qquad+
\mathbb{E} \big[\mathrm{1}_{\jac F(0)<0} \cdot \OrdoP(\rho^{11})
\,\big|\, G(z)=F(0)=0\big].
\end{align}
Indeed, by Lemma \ref{lem:minmax1}, we have that, conditionally on $G(z)=F(0)=0$,
\begin{align}\label{eq:exp} 
&|\partial G(z)|^2 \cdot |\jac F(0)| \cdot \mathrm{1}_{\jac F(0)<0} 
\\
&\qquad= 
\big(
\tfrac14 |z|^4 |\partial^3 G(0)|^2 + \OrdoP(\rho^5)\big)
\cdot
\big| |\partial^2 G(0)|^2 - \tfrac{|z|^6}{36} |\partial^3 G(0)|^2 + \OrdoP(\rho^7) \big| \cdot \mathrm{1}_{\jac F(0)<0} 
\\
&\qquad=\tfrac14 |z|^4 |\partial^3 G(0)|^2 \cdot \big| |\partial^2 G(0)|^2 - \tfrac{|z|^6}{36} |\partial^3 G(0)|^2 \big| \cdot \mathrm{1}_{\jac F(0)<0} + \mathrm{1}_{\jac F(0)<0} \cdot \OrdoP(\rho^{11}),
\end{align}
which readily gives \eqref{wes1}.

\smallskip

\noindent {\bf Step 2}. We show that, conditionally on $G(z)=F(0)=0$,
\begin{align}\label{wes2}
\mathrm{1}_{\jac F(0)<0} \leq \mathrm{1}_{|\partial^2 G(0)|^2 \leq \frac{1}{18} |z|^6 |\partial^3 G(0)|^2}+ \mathrm{1}_{|\partial^2 G(0)|^2 \leq 2\rho^6}+
\mathrm{1}_{ |E^z| \geq \rho^{-1}}.
\end{align}

By Lemma \ref{lem:minmax1}, conditionally on $G(z)=F(0)=0$,
\begin{align}
\jac F(0) = |\partial^2 G(0)|^2 - \tfrac{1}{36} |z|^6 |\partial^3 G(0)|^2 + \rho^7 E^z,
\end{align}
where $E^z = \OrdoP(1)$. If $\jac F(0)<0$, then
\[
|\partial^2 G(0)|^2 \leq \frac{1}{18} |z|^6 |\partial^3 G(0)|^2
\mbox{ or } |\partial^2 G(0)|^2 \leq 2\rho^7 |E^z|.
\]
In addition, if the last condition holds, then at least one of next two conditions will hold: 
\[
|\partial^2 G(0)|^2 \leq 2\rho^6
\mbox{ or } \rho| E^z| \geq 1. 
\]
This proves \eqref{wes2}.

\smallskip

\noindent {\bf Step 3}. We show that 
\begin{align}\label{wes3}
\mathbb{E} \big[  \mathrm{1}_{|\partial^2 G(0)|^2 \leq \frac{1}{18} |z|^6 |\partial^3 G(0)|^2}
\,\big|\, G(z)=F(0)=0\big] = \Ordo(\rho^6).
\end{align}
Recalling that $\big(G(0), \partial G(0), \tfrac{1}{\sqrt{2!}} \partial^2 G(0), \tfrac{1}{\sqrt{3!}} \partial^3 G(0) \big)$ is a standard complex vector, we see that the covariance of
\begin{equation} \label{eq:der_vector}
\big(G(z), F(0), \partial^2G(0), \partial^3G(0) \big).
\end{equation}
is
\[
\begin{pmatrix}
1 & 0 & 0 &0 \\
0 & 1 & 0 & 0 \\
0 & 0 & 2 & 0 \\
0 & 0& 0 & 3! 
\end{pmatrix}+ \Ordo(\rho).
\]
Hence, for sufficiently small $\rho$, the covariance of $\big( \partial^2 G(0), \partial^3 G(0)\big)$ conditioned on $G(z)=F(0)=0$ is 
\begin{align}\label{eq_af}
\begin{pmatrix} 2 & 0 \\ 0 & 6 \end{pmatrix} + \Ordo(\rho)
\end{align}
and is bounded above and below by positive multiples of the identity matrix (in the Loewner order). Therefore, if we let $(Z_1,Z_2)$ be a standard complex vector and set $r=\sqrt{\tfrac{1}{18}} |z|^3$, we have
\begin{align}
&\mathbb{E} \big[  \mathrm{1}_{|\partial^2 G(0)|^2 \leq \frac{1}{18} |z|^6 |\partial^3 G(0)|^2}
\,\big|\, G(z)=F(0)=0\big]
\asymp
\mathbb{E}[\mathrm{1}_{|Z_1| < r |Z_2|}] 
\\
&\quad= \int_{\mathbb{C}^2} \mathrm{1}_{|z_1| \leq r |z_2|} \frac{1}{\pi^2} e^{-|z_1|^2} e^{-|z_2|^2} dA(z_1) dA(z_2)
\\
&\quad= 4\int_0^\infty \int_0^{r k_2} k_1 k_2 \mathrm{1}_{k_1 < r k_2} e^{-k_1^2 -k_2^2} dk_1 dk_2 = \frac{r^2}{1+r^2} \asymp r^2 \asymp |z|^6 \leq \rho^6,
\end{align}
as long as $\rho$ is sufficiently small. (See Lemma \ref{lem_explicit} or \cite{hor25work} for a computation of the integral above.)

We note that the constant $\frac{1}{18}$ played no special role in the proof of \eqref{wes3}; hence, we also have
\begin{align}\label{wes3a}
\mathbb{E} \big[  \mathrm{1}_{|\partial^2 G(0)|^2 \leq \frac{1}{36} |z|^6 |\partial^3 G(0)|^2}
\,\big|\, G(z)=F(0)=0\big] = \Ordo(\rho^6).
\end{align}
\smallskip
\noindent {\bf Step 4}. Considering again the asymptotic form \eqref{eq_af} of the covariance of $\big( \partial^2 G(0), \partial^3 G(0)\big)$ conditioned on $G(z)=F(0)=0$, we can let $Z$ be a standard complex vector, set $r=\sqrt{2}\rho^3$, and estimate
$\mathbb{E} \big[  \mathrm{1}_{|\partial^2 G(0)|^2 \leq 2\rho^6}
\,\big|\, G(z)=F(0)=0\big] \asymp \mathbb{E} \big[ |Z| \leq r\big]$ to conclude that
\begin{align}\label{wes4}
\mathbb{E} \big[  \mathrm{1}_{|\partial^2 G(0)|^2 \leq 2\rho^6}
\,\big|\, G(z)=F(0)=0\big] = \Ordo(\rho^6).
\end{align}
In addition, since $E^z = \OrdoP(1)$ we have that
\begin{align}\label{wes5}
\mathbb{E} \big[  \mathrm{1}_{ |E^z| \geq \rho^{-1}}
\,\big|\, G(z)=F(0)=0\big] = \Ordo(\rho^6).
\end{align}
(Here, the estimate also holds with any exponent in lieu of $6$.)

\smallskip
\noindent {\bf Step 5}. As a consequence of \eqref{wes2}, \eqref{wes3}, \eqref{wes4} and \eqref{wes5} we have that
\begin{align}\label{eq_lplp}
\mathbb{E} \big[  \mathrm{1}_{\jac F(0)<0}
\,\big|\, G(z)=F(0)=0\big] = \Ordo(\rho^6).
\end{align}
We let $p>1$, set $1/p+1/q=1$ and use \eqref{eq_lplp} to bound the error term in \eqref{wes1}:
\begin{align}
&\mathbb{E} \big[\mathrm{1}_{\jac F(0)<0} \cdot \OrdoP(\rho^{11})
\,\big|\, G(z)=F(0)=0\big]
\\
&\quad
\leq
\mathbb{E} \Big( \OrdoP(\rho^{11})^q
\,\big|\, G(z)=F(0)=0\big] \Big)^{1/q}
\cdot
\mathbb{E} \Big(  \mathrm{1}_{\jac F(0)<0}
\,\big|\, G(z)=F(0)=0\big] \Big)^{1/p}
\\
&\quad \leq C_p \cdot \rho^{11} \cdot \rho^{6/p},
\end{align}
where $C_p$ is a constant that depends on $p$. Choosing $p$ close to $1$ we conclude that
\begin{align}\label{wes6}
\mathbb{E} \big[\mathrm{1}_{\jac F(0)<0} \cdot \OrdoP(\rho^{11})
\,\big|\, G(z)=F(0)=0\big] = \Ordo(\rho^{16.5}).
\end{align}
\smallskip
\noindent {\bf Step 6}. Considering \eqref{wes1} and \eqref{wes6}, it remains to 
estimate the effect of replacing $\jac F(0)$ by the proxy variable
\begin{align}
    X =  |\partial^2 G(0)|^2 - \tfrac{|z|^6}{36} |\partial^3 G(0)|^2.
\end{align}
More precisely, we would like to show that
\begin{align}\label{wes7}
&\mathbb{E} \big[
\tfrac14 |z|^4 \cdot |\partial^3 G(0)|^2 \cdot | X | \cdot 
\big|
\mathrm{1}_{\jac F(0)<0}
-\mathrm{1}_{X<0}
\big|
\,\big|\, G(z)=F(0)=0
\big] = \Ordo(\rho^{16.5}).
\end{align}
Consider the event $B=\{\jac F(0)<0\} \cup \{X<0\}$. By Lemma \ref{lem:minmax1}, 
\begin{align}
\jac F(0) \cdot \mathrm{1}_B = X \cdot \mathrm{1}_B + \rho^7 \cdot E^z,
\end{align}
where $E^z=\OrdoP(1)$. If $\mathrm{1}_{\jac F(0)<0} \not= \mathrm{1}_{X<0}$, then $\mathrm{1}_B=1$ and $|X| \leq \rho^7 |E^z|$, so that $\jac F(0)$ and $X$ can have different signs. Hence, we can select $p>1$ and use \eqref{wes3a} and \eqref{eq_lplp} to estimate
\begin{align}
&\mathbb{E} \big[
\tfrac14 |z|^4 \cdot |\partial^3 G(0)|^2 \cdot | X | \cdot 
\big|
\mathrm{1}_{\jac F(0)<0}
-\mathrm{1}_{X<0}
\big|
\,\big|\, G(z)=F(0)=0
\big] 
\\
&\qquad
\leq 
\rho^7 \cdot \mathbb{E} \big[
\tfrac14 |z|^4 \cdot |\partial^3 G(0)|^2 \cdot | E^z | \cdot 
\big|
\mathrm{1}_{\jac F(0)<0}
-\mathrm{1}_{X<0}
\big|
\,\big|\, G(z)=F(0)=0
\big] 
\\
&\qquad
\lesssim 
\rho^{11} \cdot \mathbb{E} \big[ |\partial^3 G(0)|^2 \cdot | E^z | \cdot 
\big|
\mathrm{1}_{\jac F(0)<0}
-\mathrm{1}_{X<0}
\big|
\,\big|\, G(z)=F(0)=0
\big] 
\\
&\qquad
\lesssim \rho^{11} \cdot\Big(
\big(\mathbb{E} \big[
\mathrm{1}_{\jac F(0)<0}
\big|
\,\big|\, G(z)=F(0)=0
\big]\big)^{1/p}+
\big(\mathbb{E} \big[
\mathrm{1}_{X<0}
\big|
\,\big|\, G(z)=F(0)=0
\big]\big)^{1/p}\Big)
\\
&\qquad \lesssim \rho^{11+6/p},
\end{align}
where the implied constant depends on $p$. Selecting $p$ close to 1 yields \eqref{wes7}.
\end{proof}
We collect all the previous work and obtain the following result, showing remarkably strong repulsion between zeros of GEF and their critical points with negative index.
\begin{prop}\label{prop_zcm}
There exist constants $c, C > 0$ such that
\begin{align}\label{eq_zcmx}
c \rho^{20} \leq 
 \mathbb{E}[\Nz \cdot \Ncm] \leq C \rho^{20}, \quad 0 < \rho < 1.
\end{align}
\end{prop}
\begin{proof}
As in the proof of Proposition \ref{prop_zcp}, it is enough to prove the claim for small $\rho$. We start by analyzing the expression in Lemma \ref{lem_minmax2}. We observe (as was done in the proof of that lemma) that the conditional covariance of $\big( \partial^2G(0), \partial^3G(0)\big)$ given that $G(z)=F(0)=0$ is 
\[
\begin{pmatrix}
2 &0 \\ 
0 & 6 
\end{pmatrix}+ \Ordo(\rho). 
\]
Hence, if $\rho$ is sufficiently small and $(Z_1,Z_2)$ is a standard complex vector, we have
\begin{align}
&\E \big[ \big| \tfrac12 z^2 \partial^3 G(0) \big|^2 \cdot \big| |\partial^2 G(0)|^2 - \tfrac{1}{36} |z|^6  \cdot |\partial^3 G(0)|^2\big| \cdot \big|
\mathrm{1}_{|\partial^2 G(0)|^2 < \tfrac{1}{36} \cdot |z|^6 \cdot |\partial^3 G(0)|^2} \,\big|\, G(z)=F(0)=0 \big]
\\
&\quad \asymp \mathbb{E}\big[ \tfrac{|z|^4} {4} \cdot |Z_2|^2 \cdot \big||Z_1|^2 - \tfrac{|z|^6}{36} \cdot |Z_2|^2 \big| \cdot \mathrm{1}_{|Z_1| < \frac{|z|^3}{6} |Z_2| }\big]
\\
          &\quad = \frac{|z|^4}{4} \cdot \mathbb{E}\big[ |Z_2|^2 \cdot \big(\tfrac{|z|^6}{36} |Z_2|^2 - |Z_1|^2\big) \cdot \mathrm{1}_{|Z_1| < \frac{|z|^3}{6} |Z_2|} \big] 
          \\
            &\quad= \frac{|z|^4}{4} \cdot \int_{\mathbb{C}^2} |z_2|^2 \big(\tfrac{|z|^6}{36} |z_2|^2 -|z_1|^2\big) \cdot \mathrm{1}_{|z_1| < \frac{|z|^3}{6} |z_2|} \frac{1}{\pi^2}e^{-|z_1|^2 - |z_2|^2} dA(z_1) dA(z_2) \\
            &\quad={|z|^4} \cdot \int_0^\infty \int_0^{\frac{|z|^3}{6} k_2}  k_2^2 \big(\tfrac{|z|^6}{36} k_2^2- k_1^2\big) k_1 k_2 e^{-k_1^2 -k_2^2} dk_1 dk_2 \\
            &\quad= \frac{|z|^4 |z|^{12}(54+|z|^6)}{72(36 + |z|^6)^2} \asymp |z|^{16}, 
\end{align}
as follows from an explicit computation that we provide in Lemma \ref{lem_explicit_2}, or in \cite{hor25work}. (In the first $\asymp$ we used the homogeneity of the function in question; see the proof of Lemma \ref{conddens}.)

Combining this with Lemma \ref{lem_minmax2}, we conclude that, for sufficiently small $\rho$, 
\begin{align}
&\mathbb{E} \big[ |\partial G(z)|^2 \cdot |\jac F(0)| \cdot \mathrm{1}_{\jac F(0) <0 } \,\big|\, G(z)=F(0)=0 \big] \asymp |z|^{16} + \Ordo(\rho^{16.5}).
\end{align}
Since the integrand in \eqref{eq_krb} depends only on $z-w$ --- cf. Lemma \ref{lemmashift} --- we conclude that
\begin{align}
 \mathbb{E}[\Nz \cdot \Ncm] \asymp \int_{B_\rho \times B_\rho}
  \Big(|z-w|^{16} +\Ordo(\rho^{16.5}) \Big)dA(z)dA(w) \asymp \rho^{20},
\end{align}
as claimed
\end{proof}

\section{Proof of Theorem \ref{th_main}}\label{sec_proof}
We collect all partial results to prove Theorem \ref{th_main}.

$\eqref{eq_repz}$ follows from
\cite[Theorem 1.1]{MR2885614}, 
Prop. \ref{repc} gives \eqref{eq_repc}, Prop. \ref{repp} gives \eqref{eq_repp}, Prop. \ref{repm} gives \eqref{eq_repm}, Prop. \ref{prop_zcp} gives \eqref{eq_zcp} and Prop. \ref{prop_zcm} gives \eqref{eq_zcm}.

The first order statistics of the zeros of $F$ were studied in \cite{hkr22}.
Proposition 3.2 in \cite{hkr22} --- which is an application of \cite[Proposition 6.5]{level} --- shows that the zeros of $F$ are almost surely non-degenerate (invertible Jacobi matrix). Hence,
\begin{align}
\Nc = \Ncp + \Ncm.
\end{align}
We write
\begin{align}
    \Ncm \cdot \Ncp = \tfrac{1}{2} \big(\Nc \cdot  (\Nc -1) - \Ncp \cdot (\Ncp -1 ) - \Ncm \cdot (\Ncm -1)\big),
\end{align}
and use \eqref{eq_repp} and \eqref{eq_repm} to conclude that
\begin{align}
    &\mathbb{E}[\Ncp \cdot \Ncm] = \frac{1}{2} \mathbb{E}[\Nc\cdot(\Nc - 1)] + \Ordo(\rho^7).
\end{align}
In addition, as shown in \cite[Example 1.3, Theorem 1.8]{hkr22},
\begin{align}
\tfrac{4}{3}  \rho^2 = \mathbb{E}[\Ncp] \asymp \mathbb{E}[\Ncm] =
\tfrac{1}{3}  \rho^2 
\end{align} 
and \eqref{eq_cmcp} follows from \eqref{eq_repc} as $\mathbb{E}[\Nc] = \frac{5}{3} \rho^2$.

Finally $\eqref{eq_zc}$ follows from $\eqref{eq_zcp}$ and $\eqref{eq_zcm}$ as
\begin{equation*}
    \Nz \cdot \Nc = \Nz \cdot \Ncp + \Nz \cdot \Ncm.
\end{equation*}

\qed

\begin{rem} In the situation of Theorem \ref{th_main} we also have
\begin{align}
    \mathbb{E}[\Ncm \cdot \Nc] \asymp \rho^2 \qquad \mbox{and} \qquad
    \mathbb{E}[\Ncp \cdot \Nc] \asymp \rho^2.
\end{align}
Indeed,
this follows from \eqref{eq_repp}, \eqref{eq_repm} and $\eqref{eq_cmcp}$ as
\begin{equation*}
    \Ncm \cdot \Nc = \Ncm \cdot (\Ncm + \Ncp) = \Ncm \cdot (\Ncm - 1) + \Ncm + \Ncm \cdot \Ncp,
\end{equation*}
and
\begin{equation*}
    \Ncp \cdot \Nc = \Ncp \cdot (\Ncp - 1) + \Ncp + \Ncm \cdot \Ncp.
\end{equation*}
\end{rem}

\section{Appendix}\label{sec_app}
\subsection{Kac-Rice formulas}
The following lemma uses the Kac-Rice formulas to provide intensity functions for the statistics \eqref{eq_stats}. Since the involved vector fields may become singular (degenerate covariance), a regularization argument is needed, and we provide the details here.
\begin{lemma}\label{kacrice}
For $0 < \rho \leq 1$, \eqref{kr1}, \eqref{kr2}, \eqref{kr3}, \eqref{eq_kra} and \eqref{eq_krb} hold.   
\begin{proof}
    \textbf{Step 1:} We treat first \eqref{kr1}, \eqref{kr2} and \eqref{kr3}. Let
    \begin{equation*}
        Z(z, w) = (\text{Re}(F(z)), \mathrm{Im}(F(z)), \text{Re}(F(w)), \mathrm{Im}(F(w))).
    \end{equation*}
    Then $Z$ is a Gaussian vector on ${\overline{\mathbb{D}}}^2$, which is almost surely $C^\infty$ and, under the identification $\mathbb{C} \simeq \mathbb{R}^2$, has covariance kernel given by \eqref{eq_cF}. It follows that $Z$ non-degenerate except on the diagonal $\{(z, w) \in \overline{\mathbb{D}}^2 : z = w\}$, since the covariance matrix of  $(F(z), F(w))$ has determinant
    \begin{equation}\label{eq_det1}
        e^{|z|^2 +|w|^2} (1-e^{-|z-w|^2}(1-|z-w|^2)^2),
    \end{equation}
    which is non-zero for $z\neq w$,
    cf. Section \ref{sec_prelim}. As the probability density of $F(z)$ is bounded near $0$ uniformly on $z$, we also have
    \begin{align*}
         &\mathbb{P}(\exists t : Z(t) = 0 \text{ and } \jac Z(t) = 0) =0,
    \end{align*}
    see \cite[Proposition 6.5]{level} or \cite[Proposition 3.2]{hkr22}. 
    We can then invoke \cite[Theorem 6.4]{level} and learn that for every bounded continuous function $g: \mathbb{R}^4 \times \mathbb{R}^{4\times 4} \rightarrow \mathbb{R}$ and every compact set $E \subset \overline{\mathbb{D}}^2$ such that $E \cap \{(z, w) : z = w\} = \emptyset$ we have
    \begin{equation*}
        \mathbb{E} \Big[\sum_{t \in E : Z(t) = 0} g(Z(t), DZ(t)) \Big] = \int_{E} p_{Z(t)}(0) \,\mathbb{E}\big[g(Z(t), DZ(t)) |\jac Z(t)| \,\big|\, Z(t) = 0\big] dt,
    \end{equation*}
where $\pi^2 \cdot p_{Z(t)}(0)$ is the pointwise inverse of the determinant \eqref{eq_det1}. We note that $\jac Z(z,w)= \jac F(z) \jac F(w)$ and that
    \begin{equation*}
        \mathbb{E}[\Nc \cdot (\Nc -1)] = \mathbb{E}[\sum_{\substack{(z, w) \in B_\rho^2 \\F(z) = F(w) = 0}} 1 - \sum_{\substack{z \in B_\rho \\ F(z) = 0}} 1] = \mathbb{E}[ \sum_{\substack{(z, w) \in B_\rho^2 \backslash \{z= w\} \\ F(z) = F(w) = 0}} 1],
    \end{equation*}
    \begin{align*}
        \mathbb{E}[\Ncp \cdot (\Ncp -1 )] &= \mathbb{E}[\sum_{\substack{(z, w) \in B_\rho^2 \\F(z) = F(w) = 0}} \mathrm{1}_{\jac F(z) > 0} \mathrm{1}_{\jac F(w) > 0} -\sum_{\substack{z \in B_\rho \\ F(z) = 0}} \mathrm{1}_{\jac F(z) > 0}] \\
        &= \mathbb{E}[\sum_{\substack{(z, w) \in B_\rho^2 \backslash \{z= w\} \\ F(z) = F(w) = 0}} \mathrm{1}_{\jac F(z) > 0} \mathrm{1}_{\jac F(w) > 0}]
    \end{align*}
    and similarly
    \begin{equation*}
        \mathbb{E}[\Ncm \cdot (\Ncm -1)] = \mathbb{E}[\sum_{\substack{(z, w) \in B_\rho^2 \backslash \{z= w\} \\ F(z) = F(w) = 0}} \mathrm{1}_{\jac F(z) < 0} \mathrm{1}_{\jac F(w) < 0}].
    \end{equation*}
    Thus, with $\mathcal{N}_\rho$ one of $\Nc, \Ncp, \Ncm$, we have that
    \begin{equation*}
        \mathbb{E}[\mathcal{N}_\rho \cdot (\mathcal{N}_\rho-1)] = \mathbb{E}[\sum_{\substack{(z, w) \in B_\rho^2 \backslash \{z= w\} \\ F(z) = F(w) = 0}} g(Z(t), DZ(t))],
    \end{equation*}
    where $g$ is an adequate non-negative and bounded (but not necessarily continuous) real-valued function. We proceed as in the proof of \cite[Lemma 3.3]{hkr22}, approximating $B_\rho^2 \backslash \{z= w\}$ by compact sets and $g$ by continuous functions.
    Fix $0 < \delta < \rho \leq 1$, let $E_\delta = \overline{B_{\rho - \delta}^2} \backslash \{|z-w| < \delta\}$ and let $\varphi_n$ be a sequence of non negative, bounded and continuous functions such that $ \varphi_n \uparrow g$, defined noting that
    \begin{equation*}
        \begin{cases} 1 \text{ if } x > \frac{1}{n} \\
                        nx \text{ if } 0 \leq x \leq \frac{1}{n}\\
                        0 \text{ if } x < 0 \end{cases} \uparrow \mathrm{1}_{x > 0}
    \quad\mbox{and}\quad
        \begin{cases} 1 \text{ if } x < -\frac{1}{n} \\
                        -nx \text{ if } -\frac{1}{n} \leq x \leq 0\\
                        0 \text{ if } x > 0 \end{cases} \uparrow \mathrm{1}_{x < 0}.
    \end{equation*}
    Then for every $n$, as $E_\delta \subset \overline{\mathbb{D}}^2$ is a compact set such that $E_\delta \cap \{(z, w) : z = w\} = \emptyset$, 
    \begin{align*}
        &\mathbb{E}\Big[\sum_{\substack{(z, w) \in E_\delta, Z(z, w) = 0}} \varphi_n(Z(z,w), DZ(z,w))\Big] \\
        &\quad= \int_{E_\delta}  \mathbb{E}[\varphi_n(Z(z, w), DZ(z, w)) |\jac F(z)| |\jac F(w)| \,\big| \,F(z) = F(w) = 0] \\
        &\qquad\qquad\qquad\cdot p_{F(z), F(w)}(0, 0)\,dA(z) dA(w).
    \end{align*}
    By monotone convergence, we can let $n \rightarrow +\infty$ and conclude that the same formula holds with $g$ in lieu of $\varphi_n$. We subsequently let $\delta \downarrow 0$ and apply again monotone convergence. Noting that $E_\delta \uparrow B_\rho^2 \backslash \{z= w\}$ and that $\{(z, w) : z= w\}$ has measure zero, we obtain \eqref{kr1}, \eqref{kr2}, \eqref{kr3}.

\smallskip
    
    \noindent \textbf{Step 2}. We consider \eqref{eq_kra} and \eqref{eq_krb}. Let
    \begin{equation*}
        Z(z, w) = (\text{Re}(G(z)), \mathrm{Im}(G(z)), \text{Re}(F(w)), \mathrm{Im}(F(w))).
    \end{equation*}
    Then $Z(t): \overline{\mathbb{D}}^2 \rightarrow \mathbb{R}^4$ is Gaussian, $Z$ is almost surely $C^1$ and $Z(t)$ is non degenerate for all $t$, because the covariance matrix of $(G(z), F(w))$, given by \eqref{eq_cF_2}, has determinant
    \begin{equation}\label{eq_det2}
        e^{|z|^2 +|w|^2} (1- e^{-|z-w|^2} |z-w|^2),
    \end{equation}
    which is non-zero. In addition,
    \begin{align*}
         &\mathbb{P}(\exists t : Z(t) = 0 \text{ and } \jac Z(t) = 0) \\         
         &\leq \mathbb{P}(\exists z \in \overline{\mathbb{D}} : G(z) = 0 \text{ and } \jac G(z) = 0) + \mathbb{P}(\exists z \in \overline{\mathbb{D}} : F(z) = 0 \text{ and } \jac F(z) = 0) = 0,
    \end{align*}
    as the zeros of $F$ and $G$ are almost surely non-degenerate. We invoke \cite[Theorem 6.4]{level} to conclude: for every compact set $E \subset \overline{\mathbb{D}}^2$ and every bounded continuous function $g: \mathbb{R}^4 \times \mathbb{R}^{4 \times 4} \rightarrow \mathbb{R}$ we have
    \begin{equation*}
        \mathbb{E}\Big[\sum_{t \in E, Z(t) = 0} g(Z(t), DZ(t))\Big] = \int_{E} p_{Z(t)}(0) \mathbb{E}\big[g(Z(t), DZ(t)) | \jac Z(t)| \,\big|\, Z(t) =0 \big]\, dt,
    \end{equation*}
    where $\jac Z(t)=\jac G(z)\jac F(w)=|\partial G(z)|^2 \jac F(w)$ and $\pi^2 \cdot p_{Z(t)}(0)$ is the pointwise inverse of the determinant \eqref{eq_det2}. We note that
    \begin{equation*}
        \mathbb{E}[\Nz \cdot \Ncp] = \mathbb{E}\Big[\sum_{\substack{(z, w) \in B_\rho^2 \\ G(z) = F(w) = 0}} \mathrm{1}_{\jac F(w) > 0}\Big] \quad\mbox{and}\quad
        \mathbb{E}[\Nz \cdot \Ncm] = \mathbb{E}\Big[\sum_{\substack{(z, w) \in B_\rho^2 \\ G(z) = F(w) = 0}} \mathrm{1}_{\jac F(w) < 0}\Big],
    \end{equation*}
    let $E_\delta = \overline{B_{\rho-\delta}}^2$, $0 < \delta \leq \rho$, and let $\varphi_n^\pm$ be non-negative bounded continuous functions such that
    \begin{equation*}
        \varphi_n^+(Z(z, w), DZ(z, w)) \uparrow \mathrm{1}_{\jac F(w) > 0} = g^+(Z(z,w), DZ(z,w)),
    \end{equation*}
    \begin{equation*}
        \varphi_n^-(Z(z, w), DZ(z, w)) \uparrow \mathrm{1}_{\jac F(w) < 0} = g^-(Z(z,w), DZ(z, w)).
    \end{equation*}
    Just as in Step 1, we apply monotone convergence, first as $n \rightarrow +\infty$ and then as $\delta \downarrow 0$, to obtain \eqref{eq_kra} and \eqref{eq_krb}.    
\end{proof}
\end{lemma}
\begin{lemma}[Translation invariance of certain statistics]\label{lemmashift} 
Consider the following intensity functions, defined for $z \neq w \in \mathbb{C}$:
\begin{align}
&q^c(z,w)=\frac{\mathbb{E}\big[|\jac F(z) \cdot \jac F(w)|\big| F(z) = F(w)= 0\big]}{e^{|z|^2 + |w|^2} \big(1- e^{-|z-w|^2}(1 - |z-w|^2)^2\big)  }\,,\\
&q^{c,+}(z,w)=\frac{\mathbb{E}\big[|\jac F(z) \cdot \jac F(w)|\mathrm{1}_{\jac F(z) > 0} \mathrm{1}_{\jac F(w) > 0} \big| F(z) = F(w)= 0\big]}{e^{|z|^2 + |w|^2} \big(1- e^{-|z-w|^2}(1 - |z-w|^2)^2\big)  }\,,\\
        &q^{c,-}(z,w)=\frac{\mathbb{E}\big[|\jac F(z) \cdot \jac F(w)|\mathrm{1}_{\jac F(z) < 0} \mathrm{1}_{\jac F(w) < 0} \big| F(z) = F(w)= 0\big]}{e^{|z|^2 + |w|^2} \big(1- e^{-|z-w|^2}(1 - |z-w|^2)^2\big)  }\,,\\
 &q^{z,c^+}(z,w)=\frac{\mathbb{E}\big[|\jac G(z) \cdot \jac F(w)| \mathrm{1}_{\jac F(w) > 0} \big| G(z) = F(w) = 0\big]}{e^{|z|^2 + |w|^2}\big(1- e^{-|z-w|^2} |z-w|^2\big)}\,, \\
        &q^{z,c^-}(z,w)=\frac{\mathbb{E}\big[|\jac G(z) \cdot \jac F(w)| \mathrm{1}_{\jac F(w) < 0} \big| G(z) = F(w) = 0\big]}{e^{|z|^2 + |w|^2}\big(1- e^{-|z-w|^2} |z-w|^2\big)}\,.
    \end{align}
Then $q^c(z,w)$, $q^{c,+}(z,w)$ and $q^{c,-}(z,w)$ depend on $|z-w|$ while $q^{z,c^+}(z,w)$ and $q^{z,c^-}(z,w)$ depend on $z-w$.
\end{lemma}
\begin{proof}
We use the fundamental symmetry \eqref{eq_zeta}, and note that, for any $\zeta \in \mathbb{C}$, conditionally on $F_\zeta(z) = 0$,
    \begin{align}
         F_\zeta^{(1, 0)}(z)  = e^{-\frac{|\zeta|^2}{2} + z \overline{\zeta}} F^{(1, 0)}(z-\zeta),
         \\
         F_\zeta^{(0, 1)}(z)  = e^{-\frac{|\zeta|^2}{2} + z \overline{\zeta}} F^{(0, 1)}(z-\zeta),
    \end{align}
    while, conditionally on $F_\zeta(z) = F_\zeta(w) = 0$,
    \begin{align}
        \jac F_\zeta(z) = e^{-|\zeta|^2 + 2\text{Re}(z \overline{\zeta})} \jac F(z-\zeta),
        \\
        \jac F_\zeta(w) = e^{-|\zeta|^2 + 2\text{Re}(w \overline{\zeta})} \jac F(w-\zeta),
    \end{align}
    and $F_\zeta(z) = F_\zeta(w) = 0$ iff $F(z-\zeta) = F(w-\zeta) = 0$. Thus,
    \begin{align*}
        &q^c(z,w)= \frac{\mathbb{E}\big[|\jac F_\zeta(z) \cdot \jac F_\zeta(w)|\big| F_\zeta(z) = F_\zeta(w)= 0\big]}{e^{|z|^2 + |w|^2} \big(1- e^{-|z-w|^2}(1 - |z-w|^2)^2\big)  } \\
        &\quad= \frac{e^{-2|\zeta|^2 +2 \text{Re}(z\overline{\zeta}) + 2 \text{Re}(w\overline{\zeta})}}{e^{|z|^2 + |w|^2}} \frac{\mathbb{E}\big[|\jac F(z-\zeta) \jac F(w-\zeta)| \big| F(z-\zeta) =F(w-\zeta) = 0\big]}{1-e^{-|z-w|^2} (1-|z-w|^2)^2}
        \\
        &\quad= \frac{1}{e^{|z-\zeta|^2 + |w-\zeta|^2}} \frac{\mathbb{E}\big[|\jac F(z-\zeta) \jac F(w-\zeta)| \big| F(z-\zeta) =F(w-\zeta) = 0\big]}{1-e^{-|z-w|^2} (1-|z-w|^2)^2}
        \\
        &\quad=q^c(z-\zeta,w-\zeta).
    \end{align*}
Hence $q^{c}(z,w)$ depends only on $z-w$. Similar conclusions follow for $q^{c,+}$ and $q^{c,-}$, after noting that, conditionally on $F(z-\zeta) = F(w-\zeta) = 0$,
\[\mathrm{sgn}(\jac F_\zeta(z))=\mathrm{sgn}(\jac F(z-\zeta)).\]
The argument for $q^{z,c^+}(z,w)$ and $q^{z,c^-}(z,w)$ is completely analogous. In addition, inspection of \eqref{eq_cF} shows that the stochastics of $F$ are also invariant under rotations: $F(\cdot) \stackrel{(d)}{=} F(e^{i\theta} \cdot)$. Hence, $q^{c}(z,w)$, $q^{c,+}$ and $q^{c,-}$ depend only on $|z-w|$.
\end{proof}

\begin{lemma} \label{lemmared}
Let $\sigma: \mathbb{R}^+ \rightarrow \mathbb{R}^+$ be measurable. Then  
\begin{equation*}
    \rho^2 \int_{0}^{\rho/2} \sigma(r) r dr
\lesssim 
\int_{B^2_\rho} \sigma(|z-w|) dA(z) dA(w) \lesssim \rho^2 \int_{0}^{2\rho} \sigma(r) r dr.
\end{equation*}
\end{lemma}
\begin{proof}
    Note first that
    \begin{align*}
        &\int_{B_\rho^2} \sigma(|z-w|) dA(w) dA(w) = \int_{B_\rho} \int_{B_\rho(z)} \sigma(|u|)dA(u) dA(z) \\
        &\quad =\int_{\mathbb{C}^2}  \sigma(|u|) \mathrm{1}_{z \in B_\rho} \mathrm{1}_{u \in B_\rho(z)} dA(z) dA(u)
        = \int_{B_{2\rho}} \sigma(|u|) \int_{\mathbb{C}} \mathrm{1}_{z \in B_\rho} \mathrm{1}_{z \in B_\rho(u)} dA(z) dA(u)
    \end{align*}
    and that
    \begin{equation*}
        \int_{\mathbb{C}} \mathrm{1}_{z \in B_\rho} \mathrm{1}_{z \in B_\rho(u)} dA(z) = |B_\rho(0) \cap B_\rho(u)|
    \end{equation*}
    depends only on $|u|$, and it decreases as $|u|$ increases. Hence, for all $u \in \mathbb{C}$,
    \begin{equation*}
        \int_{\mathbb{C}} \mathrm{1}_{z \in B_\rho} \mathrm{1}_{u \in B_\rho(z)} dA(z) \leq \pi \rho^2,
    \end{equation*}
    while, for $|u| \leq \frac{\rho}{2}$,
    \begin{equation*}
        \int_{\mathbb{C}} \mathrm{1}_{z \in B_\rho} \mathrm{1}_{u \in B_\rho(z)} dA(z) \geq |B_\rho(0) \cap B_\rho(\tfrac{\rho}{2})| = \rho^2 |B_1(0) \cap B_1(\tfrac{1}{2})|.
    \end{equation*}
    Since $\sigma(r) \geq 0$, we conclude that, for all $r$,
    \begin{equation*}
        \rho^2 |B_1(0) \cap B_1(\tfrac{1}{2})| \int_{B_{\frac{\rho}{2}}} \sigma(|u|) dA(u)
        \leq 
        \int_{B^2_\rho} \sigma(|z-w|) dA(z) dA(w) \leq \pi \rho^2 \int_{B_{2\rho}} \sigma(|u|) dA(u).
    \end{equation*}
    As 
    \begin{equation*}
        \int_{B_{\frac{\rho}{2}}} \sigma(|u|) dA(u) = 2\pi \int_0^{\frac{\rho}{2}} \sigma(r)r dr \qquad \mbox{and} \qquad \int_{B_{2\rho}} \sigma(|u|) dA(u) = 2\pi   \int_0^{2\rho} \sigma(r) r dr, 
    \end{equation*}    
    the result follows.
\end{proof}

\subsection{An estimate for Gaussian vectors}
The following lemma elaborates on \cite[Lemma 6]{ladgham2023local}.
\begin{lemma}\label{lemappxdirect}
    Let $Z_1, Z_2, Z_3$ be centered independent complex random vectors with common variance $\sigma>0$ and let $0 \leq \eta < 2$. Then there exists a constant $C_{\sigma,\eta}>0$ such that for $r > 0$: 
    \begin{align}\label{eq_pl1}
        \mathbb{P}(|\mathrm{Im}(Z_2 \overline{Z_1})|< r) \leq C_{\sigma,\eta}
        r (1+|\log(r)|),
    \end{align}
    \begin{align}\label{eq_pl2}        \mathbb{P}\big(|\mathrm{Im}\big(Z_2 \overline{ Z_1}\big) +r \mathrm{Im}\big(i |Z_2|^2 + \tfrac{1}{3} Z_3 \overline{ Z_1}\big)| < r^{2-\eta}\big) \leq C_{\sigma,\eta} r^{2-\eta} (1+|\log(r^{2-\eta})|),
    \end{align}
    \begin{align}\label{eq_pl3}
        \mathbb{P}\big(|\mathrm{Im}\big(Z_2 \overline{ Z_1}\big) -r \mathrm{Im}\big(i |Z_2|^2 - \tfrac{1}{3} Z_3 \overline{ Z_1}\big)| < r^{2-\eta}\big) \leq C_{\sigma,\eta} r^{2-\eta} (1+|\log(r^{2-\eta})|).
    \end{align}
\end{lemma}
\begin{proof}
Since the right-hand sides of \eqref{eq_pl1}, \eqref{eq_pl2} and \eqref{eq_pl3} are $\gtrsim 1$ as soon as $r$ is bounded away from $0$, we can focus on small $r$. We will use the following fact, which is part of the proof of \cite[Lemma 6]{ladgham2023local}: If $W_1, W_2$ are i.i.d. non-constant real Gaussian random variables, then there exists a constant $C>0$, depending on $\var(W_1)$, such that, for all $a, b, c \in \mathbb{R}$ and all $r > 0$,
    \begin{equation}\label{eq_pl4}
        \mathbb{P}(|(W_1 - a)(W_2 - b) + c| < r) \leq Cr(1+|\log(r)|).
    \end{equation} 
    
    Write $Z_j = X_j + i Y_j$ and consider first
    \begin{align}
        \mathrm{Im}(Z_2\overline{Z_1}) = X_1 Y_2 - X_2 Y_1.
    \end{align}
    Since $(X_1, Y_2)$ are i.i.d. non-constant Gaussian random variables independent of $(X_2, Y_1)$, \eqref{eq_pl1} follows from \eqref{eq_pl4} 
    with $W_1=X_1, W_2=Y_2$, $a=b=0$ and $c$= a realization of $X_2Y_1$. More precisely:
      \begin{align*}
         \mathbb{P}(|\mathrm{Im}(Z_2 \overline{Z_1})|< r) &= \mathbb{E}\big[\mathbb{P}(|X_1 Y_2 - X_2 Y_1|< r\, \big|\,X_2, Y_1)\big] \\
         &\lesssim \mathbb{E}\big[r(1+|\log(r)|)\big] = r(1+|\log(r)|).
    \end{align*}
    Second, we consider
    \begin{align}
        &\mathrm{Im}\big(Z_2 \overline{ Z_1}\big) +r \mathrm{Im}\big(i |Z_2|^2 + \tfrac{1}{3} Z_3 \overline{ Z_1}\big) \\
        &= X_1 Y_2 +  r Y_2^2 - X_2 Y_1 + r X_2^2 + \tfrac{r}{3} X_1 Y_3 - \tfrac{r}{3} X_3 Y_1 \\
        &= (Y_2 + \tfrac{r}{3} Y_3 ) (X_1+ r Y_2-\tfrac{r^2}{3} Y_3) + (\tfrac{r^3}{9} Y_3^2 - X_2 Y_1 + r X_2^2  - \tfrac{r}{3} X_3 Y_1).
    \end{align}
We note that $(X_1 + r Y_2, Y_2)$ is centered (zero expectation) and has covariance matrix 
\begin{equation}\label{eq_pms}
        \sigma \begin{pmatrix}
            \frac{1}{2} + \frac{r^2}{2} & \frac{r}{2} \\
            \frac{r}{2} & \frac{1}{2}
        \end{pmatrix}.
    \end{equation}
    Hence, for sufficiently small $r>0$, the eigenvalues $\lambda$ of \eqref{eq_pms} satisfy $\sigma/4 < \lambda < \sigma$ and the probability density of 
$(X_1 + r Y_2, Y_2)$, denoted $f_r$, satisfies $f_r \lesssim g$, where $g$ is the density of $(W_1, W_2) \sim \mathcal{N}_{\mathbb{C}}(0,\sigma^2)$. Thus,
   \begin{align*}
       & \mathbb{P}\big(|\mathrm{Im}\big(Z_2 \overline{ Z_1}\big) +r \mathrm{Im}\big(i |Z_2|^2 + \tfrac{1}{3} Z_3 \overline{ Z_1}\big)| < r^{2-\eta}\big) \lesssim \mathbb{P}\big(|(W_2 - \tfrac{r^2}{3} Y_3)(W_1 + \tfrac{r}{3} Y_3) + Z| < r^{2-\eta}\big),
    \end{align*}
    where $Z=\tfrac{r^3}{9} Y_3^2 - X_2 Y_1 + r X_2^2  - \tfrac{r}{3} X_3 Y_1$ and $Y_3$ are independent of $(W_1,W_2)$. Hence, as before, we can apply \eqref{eq_pl4}, after conditioning to $Y_3, Z$ to get \eqref{eq_pl2}.

    Lastly, we consider
    \begin{align*}
        &\mathrm{Im}(Z_2 \overline{Z_1} - r \mathrm{Im}(i |Z_2|^2 - \tfrac{1}{3} Z_3 \overline{Z_1}) \\
        &= X_1 Y_2 - r Y_2^2 - X_2 Y_1  - r X_2^2 + \tfrac{r}{3} X_1 Y_3 - \tfrac{r}{3} X_3 Y_1 \\
        &= (Y_2 + \tfrac{r}{3} Y_3)(X_1 - r Y_2 + \tfrac{r^2}{3} Y_3) + (-\tfrac{r^3}{9} Y_3^2 - X_2 Y_1 - r X_2^2 -\tfrac{r}{3} X_3 Y_1).
    \end{align*}
    The vector $(X_1 - r Y_2, Y_2)$ is independent of $(X_2, X_3, Y_1, Y_3)$, it is  centered and has covariance matrix given by
    \begin{equation*}
        \sigma \begin{pmatrix}
            \frac{1+r^2}{2} & -\frac{r}{2} \\
            -\frac{r}{2} & \frac{1}{2}
        \end{pmatrix},
    \end{equation*}
    which has, for sufficiently small $r$, eigenvalues $\lambda$ satisfying
    $\sigma/4 < \lambda < \sigma$. Hence, \eqref{eq_pl3} follows with the same argument used for \eqref{eq_pl2}.
\end{proof}    

\subsection{Computation of covariances}
The following computations can also be followed with the symbolic notebook \cite{hor25work}.

\begin{lemma}\label{matrixcomp}
The covariance kernel of $F$ is given by \eqref{eq_cF}, the covariance of $F(z)$ and $G(w)$ is given by \eqref{eq_cF_2},
and, for $r \geq 0$, the covariance matrix of the vector \[(F(ir), F(-ir), F^{(1, 0)}(0), F^{(0, 2)}(0), F^{(0, 3)}(0))\] is given by \eqref{eq_mtotr}.
\end{lemma}
\begin{proof}
Since $G$ is almost surely smooth and has a smooth covariance, we can compute correlations between derivatives of $G$ by exchanging differentiation and expectation in \eqref{eq_kergef}, see, e.g., \cite[Chapter 1]{level}. Recalling that $F(z) = \overline{z} G(z) - \partial G(z)$ we get
\begin{align}
        \mathbb{E}[F(z)\overline{G(w)}] &= \overline{z} \mathbb{E}[G(z)\overline{G(w)}] - \partial_z \mathbb{E}[G(z)\overline{G(w)}]
        = \overline{z} e^{z\overline{w}} - \overline{w} e^{z \overline{w}}.
\end{align}
Hence \eqref{eq_cF_2} holds. Similarly,
    \begin{align}
        \mathbb{E}[F(z)\overline{F(w)}] &= \overline{z}w \mathbb{E}[G(z)\overline{G(w)}] - \overline{z} \mathbb{E}[G(z)\overline{\partial G(w)}] - w \mathbb{E}[\partial G(z)\overline{G(w)}] + \mathbb{E}[\partial G(z)\overline{\partial G(w)}] \\
        &= (\overline{z} w - \overline{z} \overline{\partial}_w - w \partial_z + \partial_z \overline{\partial}_w) e^{z\overline{w}} = (\overline{z}w - z\overline{z} -w \overline{w} + 1 + z\overline{w}) e^{z\overline{w}} \\
        &= (1-|z-w|^2) e^{z\overline{w}}.
    \end{align}
    Hence \eqref{eq_cF} holds and we can compute
    \begin{align}
        \mathbb{E}[F(ir) \overline{F(ir)}]= \mathbb{E}[F(-ir)\overline{F(-ir)}] = e^{r^2},
    \end{align}
    and
    \begin{align}
        \mathbb{E}[F(ir)\overline{F(-ir)}] = (1-4r^2) e^{-r^2}
    \end{align}
    from which \eqref{eq_m1r} follows.
    We next write
    \begin{align}
        F(z) = (x- i y) G(x+ i y)  - \partial G(x+ i y), \qquad z = x + iy
    \end{align}
    and use the analyticity of $G$ to compute
    \begin{align}
        F^{(1, 0)}(z) = G(x + i y) + \overline{z} \partial G(z) - \partial^2 G(x+ i y),
    \end{align}
    \begin{align}
        F^{(0 ,2)}(z) &= \partial_y^2 ((x- i y) G(x+ i y) - \partial G(x+ i y)) \\    
        &=\partial_y( - i G(x + i y) + (x - i y) i \partial G(x + i y) - i\partial^2 G(x + i y)) \\
        &= 2 \partial G(x + i y)  - (x - i y) \partial^2  G(x + i y) + \partial^3 G(x + i y),
    \end{align}
    and
    \begin{align}
        &F^{(0, 3)}(z) = \partial_y F^{(0, 2)}(z) 
        \\
        &\quad= 2 i \partial^2 G(x + i y) + i \partial^2 G(x + i y)  - (x- i y)i \partial^3 G(x + i y)  + i \partial^4 G(x + i y).
    \end{align}
    In particular,
    \begin{align}
        F^{(1, 0)}(0) = G(0) - \partial^2 G(0),
    \end{align}
    \begin{align}
        F^{(0, 2)}(0) = 2 \partial G(0) + \partial^3 G(0),
    \end{align}
    and
    \begin{align}
        F^{(0, 3)}(0) = 3i \partial^2 G(0) + i \partial^4 G(0).
    \end{align}
    Replacing $F(ir) = -i rG(ir) - \partial G(ir)$, $r \in \mathbb{R}$, in the equations above gives
    \begin{align}
        F(ir)\overline{F^{(1, 0)}(0)} &= (- ir G(ir) - \partial G(ir)) \overline{(G(0) - \partial^2 G(0))} \\
        &= - ir G(ir) \overline{G(0)} - \partial G(ir) \overline{G(0)} + ir G(ir) \overline{\partial^2 G(0)} + \partial G(ir)\overline{\partial^2 G(0)},
    \end{align}
    \begin{align}
        F(ir)\overline{F^{(0, 2)}(0)} &= (-ir G(ir) - \partial G(ir))\overline{(2\partial G(0)+\partial^3 G(0))} \\
        &= - 2 ir G(ir) \overline{\partial G(0)} - 2\partial G(ir)\overline{\partial G(0)} - ir G(ir) \overline{\partial^3 G(0)} - \partial G(ir) \overline{\partial^3 G(0)}, 
    \end{align}
    and
    \begin{align}
        F(ir) \overline{F^{(0, 3)}(0)} &= (-ir G(ir) - \partial G(ir)) \overline{(3i \partial^2 G(0) +i \partial^4 G(0))} \\
        &= -3r G(ir) \overline{\partial^2 G(0)}  + 3i \partial G(ir) \overline{\partial^2 G(0)} - r G(ir) \overline{\partial^4 G(0)} + i \partial G(ir) \overline{\partial^4 G(0)}.
    \end{align}
    We note that
    \begin{align}
        \mathbb{E}[G(z) \overline{\partial^k G(0)}] = \overline{\partial}_w^k e^{z \overline{w}}|_{w = 0} =  z^k
    \end{align}
    and
    \begin{align}
        \mathbb{E}[\partial G(z) \overline{\partial^k G(0)}] =  \partial_z \overline{\partial}_w^k e^{z\overline{w}} \big|_{w = 0}= (k z^{k-1} e^{z\overline{w}} + z^k \overline{w} e^{z\overline{w}})\big|_{w = 0} = k z^{k-1},
    \end{align}
    where, for $k=0$, we interpret $k z^{k-1}=0$ for all $z\in \mathbb{C}$. We use these expressions to compute
    \begin{align}
        \mathbb{E}[F(ir) \overline{F^{(1, 0)}(0)}] = - ir  -0 + (ir)^3 + 2(ir) = ir - i r^3,
    \end{align}
    \begin{align}
        \mathbb{E}[F(ir) \overline{F^{0, 2}(0)}]  = - 2ir (ir) - 2 - ir (ir)^3 - 3 (ir)^2 = -2 + 5 r^2 - r^4
    \end{align}
    and
    \begin{align}
        \mathbb{E}[F(ir) \overline{F^{(0, 3)}(0)}] = - 3r (ir)^2 + 3i 2(ir) - r (ir)^4 +i 4 (ir)^3 = -6r + 7 r^3 - r^5,
    \end{align}
    from which \eqref{eq_m2r} follows.

Finally, we recall that $\big(G(0), \partial G(0), \tfrac{1}{\sqrt{2!}} \partial^2 G(0), \tfrac{1}{\sqrt{3!}} \partial^3 G(0) \big)$ is a standard complex vector and compute
    \begin{align}
        &\mathbb{E}[F^{(1, 0)}(0) \overline{F^{(1, 0)}(0)}] = 1 + 2! = 3, 
        \qquad
        &\mathbb{E}[F^{(1, 0)}(0) \overline{F^{(0, 2)}(0)}] = 0,
        \\
        &\mathbb{E}[F^{(1, 0)}(0) \overline{F^{(0, 3)}(0)}] = 3i\cdot 2! =  6i,
        \qquad
        &\mathbb{E}[F^{(0, 2)}(0) \overline{F^{(0, 2)}(0)}]=  4 + 3! = 10,
        \\
        &\mathbb{E}[F^{(0, 2)}(0)\overline{F^{(0 ,3)}(0)}] = 0,
    \qquad
        &\mathbb{E}[F^{(0 ,3)}(0) \overline{F^{(0 ,3)}(0)}] = 9 \cdot 2! + 4! = 42,
    \end{align}
    from which \eqref{eq_m3r} follows.
\end{proof}
\subsection{Some integrals}
The following elementary computations can also be followed with the symbolic notebook \cite{hor25work}.

\begin{lemma}\label{lem_explicit}
    \begin{align}
            \frac{1}{\pi^2} \int_{\mathbb{C}^2} \mathrm{1}_{|z_1| \leq r |z_2|}  e^{-|z_1|^2} e^{-|z_2|^2} dA(z_1) dA(z_2) = \frac{r^2}{1+r^2}.
        \end{align}
\end{lemma}
\begin{proof}
    We compute directly
    \begin{align*}
        &\int_{\mathbb{C}^2} \mathrm{1}_{|z_1| \leq r |z_2|} \frac{1}{\pi^2} e^{-|z_1|^2} e^{-|z_2|^2} dA(z_1) dA(z_2)= 4\pi^2 \int_{0}^\infty \int_0^{r k_2} \frac{1}{\pi^2} e^{-k_1^2} e^{-k_2^2} k_1 k_2 dk_1 dk_2  \\
        &\qquad= 4 \int_{0}^\infty k_2 e^{-k_2^2} \int_0^{r k_2} k_1 e^{-k_1^2} dk_1 dk_2
        = 4 \int_0^\infty k_2 e^{-k_2^2} \big[-\frac{1}{2} e^{-k_1^2}\big]_0^{r k_2} dk_2 \\
        &\qquad= 2\int_0^\infty k_2 e^{-k_2^2} (1- e^{- r^2 k_2^2}) dk_2
        = 2\Big[\frac{1}{2} \frac{e^{-k_2^2}(-1+e^{-k_2^2 r^2} - r^2)}{1+r^2}\Big]_0^\infty = \frac{r^2}{1+r^2}.\qedhere
         \end{align*}
\end{proof}
    \begin{lemma}\label{lem_explicit_2}
        \begin{align*}
            \int_{\mathbb{C}^2} |z_2|^2 \big(\tfrac{|z|^6}{36} |z_2|^2 -|z_1|^2\big) \mathrm{1}_{|z_1| < \tfrac{|z|^3}{6} |z_2|} \tfrac{1}{\pi^2}e^{-|z_1|^2 - |z_2|^2} dA(z_1) dA(z_2) = \frac{|z|^{12}(54+|z|^6)}{18(36 + |z|^6)^2}.
        \end{align*}
    \end{lemma}
    \begin{proof}
        We note that
        \begin{align*}
            &\int_{\mathbb{C}^2} |z_2|^2 (\frac{|z|^6}{36} |z_2|^2 -|z_1|^2) \mathrm{1}_{|z_1| < \frac{|z|^3}{6} |z_2|} \frac{1}{\pi^2}e^{-|z_1|^2 - |z_2|^2} dA(z_1) dA(z_2) \\
            &= 4\pi^2 \int_0^\infty \int_0^{\frac{|z|^3}{6} k_2} k_2^2 (\frac{|z|^6}{36} k_2^2 - k_1^2) \frac{1}{\pi^2} e^{-k_1^2 - k_2^2} k_1 k_2 dk_1 dk_2 \\
            &=4\int_0^\infty k_2^3 e^{-k_2^2} \int_0^{\frac{|z|^3}{6} k_2} (\frac{|z|^6}{36} k_2^2 - k_1^2) e^{-k_1^2} k_1dk_1 dk_2, 
        \end{align*}
        and compute
        \begin{equation*}
            \int_0^{\frac{|z|^3}{6} k_2} (\frac{|z|^6}{36} k_2^2 - k_1^2) e^{-k_1^2} k_1  dk_1 = [\frac{1}{2} k_2 e^{-k_1^2}(1+k_1^2 - \frac{|z|^6}{36} k_2^2) ]_0^{\frac{|z|^3}{6} k_2} = \frac{1}{2} e^{-\frac{|z|^6}{36} k_2^2} - \frac{1}{2} + \frac{|z|^6}{72} k_2^2.
        \end{equation*}
        It remains to compute
        \begin{equation*}
            \int_0^\infty k_2^3 e^{-k_2^2} (\frac{1}{2} e^{-\frac{|z|^6}{36} k_2^2} - \frac{1}{2} + \frac{|z|^6}{72} k_2^2) dk_2.
        \end{equation*}
        Noting that
        \begin{equation*}
            \int_0^{\infty} k_2^3 e^{-(1+ \frac{|z|^6}{36}) k_2^2} dk_2 = [-e^{-(1+ \frac{|z|^6}{36}) y^2}\frac{1+(1+ \frac{|z|^6}{36}) k_2^2}{2 (1+ \frac{|z|^6}{36})^2}]_0^\infty = \frac{1}{2(1+ \frac{|z|^6}{36})^2},
        \end{equation*}
        \begin{equation*}
            \int_0^{\infty} k_2^3 e^{-k_2^2} dx = [-\frac{1}{2} e^{-k_2^2} (1+k_2^2)]_0^\infty = \frac{1}{2},
        \end{equation*}
        \begin{equation*}
            \int_0^\infty k_2^5 e^{-k_2^2} dk_2 = [-\frac{1}{2}e^{-k_2^2}(2+2k_2^2+k_2^4)]_0^\infty = 1,
        \end{equation*}
        we get
        \begin{align*}
            \int_0^\infty k_2^3 e^{-k_2^2} \int_0^{\frac{|z|^3}{6} k_2} (\frac{|z|^6}{36} k_2^2 - k_1^2) e^{-k_1^2} k_1dk_1 dk_2  &= \int_0^\infty k_2^3 e^{-k_2^2} (\frac{1}{2} e^{-\frac{|z|^6}{36} k_2^2} - \frac{1}{2} + \frac{|z|^6}{72} k_2^2) dk_2 \\
            &= \frac{1}{4(1+\frac{|z|^6}{36})^2} - \frac{1}{4} + \frac{|z|^6}{72} = \frac{|z|^{12}(54 +|z|^6)}{72(36 + |z|^6)^2},
        \end{align*}
        and the result follows.
    \end{proof}


\begin{thebibliography}{10}
	
	\bibitem{adler}
	R.~J. Adler and J.~E. Taylor.
	\newblock {\em Random fields and geometry}.
	\newblock Springer Monographs in Mathematics. Springer, New York, 2007.
	
	\bibitem{MR4426161}
	J.-M. Aza\"{\i}s and C.~Delmas.
	\newblock Mean number and correlation function of critical points of isotropic
	{G}aussian fields and some results on {GOE} random matrices.
	\newblock {\em Stochastic Process. Appl.}, 150:411--445, 2022.
	
	\bibitem{level}
	J.-M. Aza\"{\i}s and M.~Wschebor.
	\newblock {\em Level sets and extrema of random processes and fields}.
	\newblock John Wiley \& Sons, Inc., Hoboken, NJ, 2009.
	
	\bibitem{MR2966361}
	J.~Baber.
	\newblock Scaled correlations of critical points of random sections on
	{R}iemann surfaces.
	\newblock {\em J. Stat. Phys.}, 148(2):250--279, 2012.
	
	\bibitem{MR4047541}
	R.~Bardenet, J.~Flamant, and P.~Chainais.
	\newblock On the zeros of the spectrogram of white noise.
	\newblock {\em Appl. Comput. Harmon. Anal.}, 48(2):682--705, 2020.
	
	\bibitem{bh}
	R.~Bardenet and A.~Hardy.
	\newblock Time-frequency transforms of white noises and {G}aussian analytic
	functions.
	\newblock {\em Appl. Comput. Harmon. Anal.}, 50:73--104, 2021.
	
	\bibitem{MR3947635}
	D.~Beliaev, V.~Cammarota, and I.~Wigman.
	\newblock Two point function for critical points of a random plane wave.
	\newblock {\em Int. Math. Res. Not. IMRN}, (9):2661--2689, 2019.
	
	\bibitem{MR4187723}
	D.~Beliaev, V.~Cammarota, and I.~Wigman.
	\newblock No repulsion between critical points for planar {G}aussian random
	fields.
	\newblock {\em Electron. Commun. Probab.}, 25:Paper No. 82, 13, 2020.
	
	\bibitem{buckley2018winding}
	J.~Buckley and N.~Feldheim.
	\newblock The winding of stationary {G}aussian processes.
	\newblock {\em Probab. Theory Related Fields}, 172(1-2):583--614, 2018.
	
	\bibitem{buckley2017fluctuations}
	J.~Buckley and M.~Sodin.
	\newblock Fluctuations of the increment of the argument for the {G}aussian
	entire function.
	\newblock {\em J. Stat. Phys.}, 168(2):300--330, 2017.
	
	\bibitem{MR2104882}
	M.~R. Douglas, B.~Shiffman, and S.~Zelditch.
	\newblock Critical points and supersymmetric vacua. {I}.
	\newblock {\em Comm. Math. Phys.}, 252(1-3):325--358, 2004.
	
	\bibitem{efkr24}
	L.~A. Escudero, N.~Feldheim, G.~Koliander, and J.~L. Romero.
	\newblock Efficient {C}omputation of the {Z}eros of the {B}argmann {T}ransform
	{U}nder {A}dditive {W}hite {N}oise.
	\newblock {\em Found. Comput. Math.}, 24(1):279--312, 2024.
	
	\bibitem{fhkr24}
	N.~Feldheim, A.~Haimi, G.~Koliander, and J.~L. Romero.
	\newblock Hyperuniformity and non-hyperuniformity of zeros of Gaussian
	Weyl-Heisenberg functions.
	\newblock {\em arXiv preprint:2406.20003}.
	
	\bibitem{MR3937291}
	R.~Feng.
	\newblock Correlations between zeros and critical points of random analytic
	functions.
	\newblock {\em Trans. Amer. Math. Soc.}, 371(8):5247--5265, 2019.
	
	\bibitem{7180335}
	P.~Flandrin.
	\newblock Time–frequency filtering based on spectrogram zeros.
	\newblock {\em IEEE Signal Processing Letters}, 22(11):2137--2141, 2015.
	
	\bibitem{7869100}
	P.~Flandrin.
	\newblock The sound of silence: Recovering signals from time-frequency zeros.
	\newblock In {\em 2016 50th Asilomar Conference on Signals, Systems and
		Computers}, pages 544--548, 2016.
	
	\bibitem{flandrin2018explorations}
	P.~Flandrin.
	\newblock {\em Explorations in time-frequency analysis}.
	\newblock Cambridge University Press, 2018.
	
	\bibitem{MR1690355}
	P.~J. Forrester and G.~Honner.
	\newblock Exact statistical properties of the zeros of complex random
	polynomials.
	\newblock {\em J. Phys. A}, 32(16):2961--2981, 1999.
	
	\bibitem{gardner2006sparse}
	T.~J. Gardner and M.~O. Magnasco.
	\newblock Sparse time-frequency representations.
	\newblock {\em Proc. Nat. Acad. Sc.}, 103(16):6094--6099, 2006.
	
	\bibitem{hkr22}
	A.~Haimi, G.~Koliander, and J.~L. Romero.
	\newblock Zeros of {G}aussian {W}eyl-{H}eisenberg functions and hyperuniformity
	of charge.
	\newblock {\em J. Stat. Phys.}, 187(3):Paper No. 22, 41, 2022.
	
	\bibitem{hor25work}
	A.~Haimi, L.~Odelius, and J.~L. Romero.
	\newblock Jupyter notebook {HOD25}.
	\newblock
	\url{https://mybinder.org/v2/gh/lukasod/LocalRepulsion-repo/main?urlpath=%2Fdoc%2Ftree%2FLocalRepulsion.ipynb},
	2025.
	
	\bibitem{MR3340325}
	B.~Hanin.
	\newblock Correlations and pairing between zeros and critical points of
	{G}aussian random polynomials.
	\newblock {\em Int. Math. Res. Not. IMRN}, (2):381--421, 2015.
	
	\bibitem{MR3689975}
	B.~Hanin.
	\newblock Pairing of zeros and critical points for random polynomials.
	\newblock {\em Ann. Inst. Henri Poincar\'e{} Probab. Stat.}, 53(3):1498--1511,
	2017.
	
	\bibitem{MR1662451}
	J.~H. Hannay.
	\newblock The chaotic analytic function.
	\newblock {\em J. Phys. A}, 31(49):L755--L761, 1998.
	
	\bibitem{gafbook}
	J.~B. Hough, M.~Krishnapur, Y.~Peres, and B.~Vir\'{a}g.
	\newblock {\em Zeros of {G}aussian analytic functions and determinantal point
		processes}, volume~51 of {\em University Lecture Series}.
	\newblock American Mathematical Society, Providence, RI, 2009.
	
	\bibitem{ladgham2023local}
	S.~Ladgham and R.~Lachieze-Rey.
	\newblock Local repulsion of planar {G}aussian critical points.
	\newblock {\em Stochastic Process. Appl.}, 166:Paper No. 104221, 25, 2023.
	
	\bibitem{MIRAMONT2024109250}
	J.~M. Miramont, F.~Auger, M.~A. Colominas, N.~Laurent, and S.~Meignen.
	\newblock Unsupervised classification of the spectrogram zeros with an
	application to signal detection and denoising.
	\newblock {\em Signal Processing}, 214:109250, 2024.
	
	\bibitem{miramont2024benchmarking}
	J.~M. Miramont, R.~Bardenet, P.~Chainais, and F.~Auger.
	\newblock Benchmarking multi-component signal processing methods in the
	time-frequency plane.
	\newblock {\em Preprint, arXiv:2402.08521}, 2024.
	
	\bibitem{NSwhat}
	F.~Nazarov and M.~Sodin.
	\newblock What is{$\ldots$}a {G}aussian entire function?
	\newblock {\em Notices Amer. Math. Soc.}, 57(3):375--377, 2010.
	
	\bibitem{MR2885614}
	F.~Nazarov and M.~Sodin.
	\newblock Correlation functions for random complex zeroes: strong clustering
	and local universality.
	\newblock {\em Comm. Math. Phys.}, 310(1):75--98, 2012.
	
	\bibitem{MR2346279}
	F.~Nazarov, M.~Sodin, and A.~Volberg.
	\newblock Transportation to random zeroes by the gradient flow.
	\newblock {\em Geom. Funct. Anal.}, 17(3):887--935, 2007.
	
	\bibitem{pascal2024point}
	B.~Pascal and R.~Bardenet.
	\newblock Point processes and spatial statistics in time-frequency analysis.
	\newblock {\em arXiv preprint:2402.19172}, 2024.
	
	\bibitem{MR1783614}
	M.~Sodin.
	\newblock Zeros of {G}aussian analytic functions.
	\newblock {\em Math. Res. Lett.}, 7(4):371--381, 2000.
	
\end{thebibliography}
\end{document}